\documentclass[12pt]{amsart}
\usepackage{amsmath,amsthm,amsfonts,amssymb,verbatim}
\usepackage{hyperref}
\usepackage[usenames]{color} 

\definecolor{teal}{RGB}{22, 180, 180}

\newtheorem{thm}{Theorem}[section]
\newtheorem{lemma}[thm]{Lemma}
\newtheorem{cor}[thm]{Corollary}
\newtheorem{proposition}[thm]{Proposition}
\newtheorem{prop}[thm]{Proposition}
\newtheorem{question}[thm]{Question}

\usepackage{pdfsync}
\usepackage{graphicx}

\theoremstyle{definition}  

\numberwithin{equation}{section}
\newtheorem{defn}[thm]{Definition}
\newtheorem{definition}[thm]{Definition}

\newtheorem{rem}[thm]{Remark}
\newtheorem{ex}[thm]{Example}
\newtheorem{exs}[thm]{Examples}
\newtheorem{remark}[thm]{Remark}

\theoremstyle{definition}
\theoremstyle{remark}

\newcounter{enumitemp}

\DeclareMathOperator{\Inert}{Inert}
\DeclareMathOperator{\lo}{o}

\DeclareMathOperator{\rr}{range}
\DeclareMathOperator{\Span}{span}

\newcommand{\am}{{\alpha_m}}

\newcommand{\A}{{\mathcal A}}

\newcommand{\C}{{\mathcal C}}

\newcommand{\E}{{\mathcal E}}
\newcommand{\F}{{\mathcal F}}
\newcommand{\cI}{{\mathcal I}}

\newcommand{\U}{{\mathcal U}}

\newcommand{\R}{\mathbb R}

\newcommand{\cN}{{\mathcal N}}
\newcommand{\N}{{\mathbb N}}

\newcommand{\ra}{{\rangle}}
\newcommand{\la}{{\langle}}

\newcommand{\Z}{\mathbb Z}

\def\la{{\langle}}

\def\ra{{\rangle}}
\def\la{{\langle}}

\def\z2s{{$\Z^2$-subshift}}
\def\zds{{$\Z^d$-subshift}}

\def\CK{{\mathcal K}}
\def\CL{{\mathcal L}}
\def\CS{{\mathcal S}}
\def\CT{{\mathcal T}}
\def\Z{{\mathbb Z}}

\newcommand{\trycomment}[1]{}

\DeclareMathOperator{\Aut}{{\rm Aut}}
\DeclareMathOperator{\End}{{\rm End}}
\DeclareMathOperator{\id}{{\rm Id}}
\DeclareMathOperator{\topo}{{\rm top}}

\title{The spacetime of a shift endomorphism}
\author{Van Cyr}
\address{Bucknell University, Lewisburg, PA 17837 USA}
\email{van.cyr@bucknell.edu}
\author{John Franks}
\address{Northwestern University, Evanston, IL 60208 USA}
\email{j-franks@northwestern.edu}
\author{Bryna Kra}
\address{Northwestern University, Evanston, IL 60208 USA}
\email{kra@math.northwestern.edu}

\subjclass[2010]{}

\thanks{The third author was partially supported by NSF grant 1500670.}

\begin{document}
\maketitle
 \begin{abstract}
The automorphism group of a one dimensional shift space over a finite
alphabet exhibits different types of behavior: for a large class with positive entropy, it contains a rich collection of
subgroups, while for many shifts of zero entropy, there are strong constraints on 
the automorphism group.  We view this from a different perspective, considering 
a single automorphism (and sometimes endomorphism) and studying the naturally associated
two dimensional shift system.  In particular, we describe the relation 
between nonexpansive subspaces in this two dimensional system and 
dynamical properties of an automorphism of the shift. 
 \end{abstract}

\section{Introduction}

Suppose $\Sigma$ is a finite alphabet and $X\subset\Sigma^{\Z}$ is a closed set that is invariant under the left shift $\sigma\colon\Sigma^{\Z}\to\Sigma^{\Z}$.  The collection of automorphisms $\Aut(X, \sigma)$, 
consisting of all homeomorphisms $\phi\colon X\to X$ that commute with $\sigma$, forms  a group (under composition).  
A useful approach to understanding a countable group $G$ is knowing if
it has subgroups which are isomorphic to (or are homomorphic images
of) simpler groups which are relatively well understood, such as
matrix groups, and in particular, lattices in classical Lie groups.
While the automorphism group of a shift is necessarily countable (as
an immediate corollary of the Curtis-Hedlund-Lyndon Theorem~\cite{H},
any automorphism $\phi\colon X\to X$ is given by a block code), there are numerous results in
the literature showing that the automorphism group of the full shift,
and more generally any mixing shift of finite type, contains
isomorphic copies of many groups: this collection includes, for
example, any finite group, the direct sum of countably many copies of
$\Z$, the free group on any finite number of generators, and the
fundamental group of any $2$-manifold (see~\cite{H, BLR, KR1}).  In
light of these results, it is natural to ask if there is any finitely
generated (or even countable) group which fails to embed in any such
automorphism group, meaning any group of the form $\Aut(X, \sigma)$.
A partial answer is given in~\cite{BK}, where it is shown that if
$(X,\sigma)$ is a subshift of finite type then any group that embeds
in the automorphism group must be residually finite.  At the other end
of the complexity spectrum for $(X,\sigma)$, there has been recent
work showing that $\Aut(X,\sigma)$ is significantly more tame for a
shift with very low complexity (see for example~\cite{CK1, CK2,
  DDMP}).

Instead of viewing the entire group, we focus on the structure inherent in
a single automorphism $\phi\in \Aut(X, \sigma)$, as studied for example in~\cite{H, BK, KR1, KR}. 
 Given an
automorphism $\phi$, there is an obvious way to associate a
$\Z^2$-shift action, which we call the spacetime of $\phi$ (in a slightly different setting, 
this is called the complete history by Milnor~\cite{milnor} and is referred to as the spacetime diagram in the cellular automata literature).   
We make use of a particular subset
of the spacetime, dubbed the light cone, that is closely related
to the notion of causal cone discussed in~\cite{milnor}.  
We show that the light cone gives a characterization of a well studied
structural feature of a $\Z^2$-shift, namely the boundary of a
component of expansive subspaces (see~\cite{BL} and~\cite{hochman}).
In particular, in \S\ref{sec:lightcone} we show that the edges of a
light cone for $\phi$ are always nonexpansive subspaces in its
spacetime (the precise statement is given in Theorem~\ref{thm nonexpansive}).

We also provide a complement to this result:
for many \z2ss with nonexpansive subspace $L$, the system is isomorphic
to the space time of an endomorphism $\phi$ by an isomorphism which carries $L$ to an
edge of the light cone of $\phi$.  

We then use these structural results to describe obstructions to embedding in the automorphism group of a shift.  
An important concept in the study of lattices is the idea of a distortion element,
meaning an element whose powers have sublinear growth of their minimal word 
length in some (and hence any) set of generators.
In \S\ref{sec:asymptotic}, 
we introduce a notion of range distortion for automorphisms, meaning 
that the range (see Section~\ref{subsec:defs} for the definitions) of the associated  block codes of iterates of the automorphism 
grow sublinearly.  An immediate observation is  that 
if an automorphism is distorted in $\Aut(X)$ (in the group sense), then it
is also range distorted.  We also introduce  a measure of non-distortion
called the asymptotic spread $A(\phi)$ of an automorphism $\phi$ and 
show that the topological entropies of $\phi$ and $\sigma$ satisfy the inequality
\[
h_{\topo}(\phi) \le A(\phi) h_{\topo}(\sigma).
\]
This recovers an inequality of Tisseur~\cite{tisseur}; his context is more restrictive, covering 
the full shift endowed with the uniform measure.  We do not appeal to measure theoretic entropy and our statement applies to a wider class of shifts.  

This inequality proves to be useful in providing obstructions to various groups embedding in the automorphism group.
These ideas are further explored in~\cite{forthcoming}. 

\subsection*{Acknowledgement} We  thank Alejandro Maass for helpful comments and for pointing us to references~\cite{sherevesky, tisseur}, and we thank Samuel Petite for helpful conversations.  
We also thank the referee for numerous comments that improved our article.

\section{Background}

\subsection{Shift systems and endomorphisms}
\label{subsec:defs}
We assume throughout that $\Sigma$ is a finite set (which we call the {\em alphabet}) endowed 
with the discrete topology
and endow $\Sigma^\Z$ with the product
topology. 
For $x \in \Sigma^\Z$, we write $x[n] \in \Sigma$ for the value of $x$ at $n \in \Z$.

The {\em left shift} $\sigma\colon \Sigma^\Z \to \Sigma^\Z$ is defined by 
$(\sigma x)[n] = x[n+1]$, and is a homeomorphism from $\Sigma^\Z$ to itself.
We say that $(X, \sigma)$ is a {\em subshift}, or just a {\em shift} when the context is clear,  if 
$X\subset \Sigma^\Z$ is a closed set that is invariant under the
left shift $\sigma\colon\Sigma^\Z\to\Sigma^\Z$.  

{\bf Standing assumption}: 
Throughout this article, $(X,\sigma)$ denotes a shift system and we assume 
that the alphabet $\Sigma$ of $X$ is finite and that the shift $(X, \sigma)$ is infinite, meaning that $|X| = \infty$. 

\begin{defn}\label{def: endo}
An {\em endomorphism} of the shift
$(X, \sigma)$ is a continuous surjection $\phi\colon X\to X$ such that
$\phi\circ\sigma = \sigma\circ\phi$.  An {\em endomorphism} which is
invertible is called an {\em automorphism}.
The group of all automorphisms of $(X, \sigma)$ is 
denoted $\Aut(X, \sigma)$, or simply $\Aut(X)$ when $\sigma$ 
is clear from the context. The semigroup of all endomorphisms of $X$
with operation composition
is denoted $\End(X, \sigma)$, or simply $\End(X)$.  We also observe
that $\End(X,\sigma) / \la \sigma \ra$, the set of cosets of the subgroup
$\la \sigma \ra$, is naturally a semigroup with multiplication
$\phi\la \sigma \ra \psi\la \sigma \ra $ defined to be $\phi\psi \la \sigma \ra$.
\end{defn}

For an interval $[n, n+1, \dots, n+k-1] \in \Z$ and $x \in X$, we let
$x[n, \dots, n+k-1]$ denote the element $a$ of $\Sigma^{k}$ with
$a_j = x[n+j]$ for $j = 0, 1, \ldots, k-1$.
Define the {\em words $\CL_k(X)$ of length $k$  in $X$} to 
be the collection of all $[a_1, \dots, a_k] \in \Sigma^{k}$ such that
there exist $x \in X$ and $m\in \Z$ with $x[m+i] = a_i$ for $1 \le i \le k$.
The length of a word $w \in \CL(X)$ is denoted by $|w|$.
The language $\CL(X)=\bigcup_{k=1}^\infty\CL_k(X)$ is defined 
to be the collection of all finite words.

The {\em complexity} of $(X, \sigma)$ is the 
function $P_X\colon \N\to\N$ that counts the number of 
words of length $n$ in the language of $X$.  Thus 
$$P_X(n) = \big \vert\CL_n(X) \big \vert.$$ 
The exponential growth rate of the complexity is
the {\em topological entropy $h_{\topo}$} of the shift $\sigma$. Thus
\[
h_{\topo}(\sigma) = \lim_{n \to \infty} \frac{\log(P_X(n))}{n}.
\]
This is equivalent to the usual definition of topological entropy using $(n,\varepsilon)$-separated sets (see, for example~\cite{LM}).

A map $\phi\colon X\to X$ is a {\em sliding block code} if there exists $R\in\N$ such that 
for any $x,y\in X$ with $x[i] = y[i]$ for $-R \le i \le R$, we have that $\phi(x)[0] = \phi(y)[0]$.
The least $R$ such that this holds is called the {\em range} of $\phi$. 

By the Curtis-Hedlund-Lyndon Theorem~\cite{H}, any endomorphism $\phi\colon X\to X$ 
of a shift $(X,\sigma)$ is a sliding block code.
In particular, $\End(X)$ is always countable. 

\begin{defn} \label{def: conjugacy}
Suppose $(X, \sigma)$ and $(X', \sigma')$  are shifts and
$\phi \in \End(X, \sigma)$ and $\phi' \in \End(X' ,\sigma')$ 
are endomorphisms.  We say that $\phi$ and $\phi'$ are {\em conjugate endomorphisms}
if there is a homeomorphism $h\colon X \to X'$ such that
\[
h \circ \sigma = \sigma' \circ h \text{ and }
h \circ \phi = \phi' \circ h.
\]
\end{defn}

A homeomorphism $h$ satisfying these properties is a sliding block code.
If $\phi$ and $\phi'$ both lie in $\Aut(X, \sigma)$, then $\phi$ and $\phi'$ are
conjugate if and only if they are conjugate as elements of the group $\Aut(X, \sigma)$.

A shift $X$ is \emph{irreducible} 
if for all words $u, v\in \CL(X)$, there exists $w\in \CL(X)$ such that 
$uwv\in\CL(X)$.  

\begin{defn}
A shift $(X,\sigma)$ is a {\em subshift of finite
type} provided it is defined by a finite set of excluded words.
In other words, there is a finite set  $\F \subset \CL(\Sigma^\Z)$ 
such that  $x \in X$ if and only if there are no $n \in \Z$ and $k>0$ such that 
$x[n, \dots, n+k] \in \F$.
\end{defn}

We make use of the following proposition due to Bowen \cite{B}. 
A proof can be found in~\cite[Theorem 2.1.8]{LM}.

\begin{prop}\label{bowen prop}
A shift $(X, \sigma)$ is a shift of finite type if and only if there exists
$n_0 \geq 0$ such that whenever $uw, wv \in \CL(X)$ and $|w| \ge n_0$, then also
$uwv \in \CL(X)$.
\end{prop}

\subsection{Higher dimensions}
More generally, one can
consider a multidimensional shift $X\subset \Sigma^{\Z^d}$ for some
$d\geq 1$, where $X$ is a closed set (with respect to the product topology) that is invariant under the
$\Z^d$ action $(T^ux)(v) = x(u+v)$ for $u\in\Z^d$.  We refer 
to $X$ with the $\Z^d$ action as a {\em \zds} and  to $\eta \in X$ as an {\em $X$-coloring
of $\Z^d$}.

We note that we have made a slight abuse of notation in passing to the multidimensional setting by denoting 
the entries of an element $x\in X$ by $x(u)$ (where $u\in\Z^d$), rather than $x[u]$ as we did for a one dimensional shift.  This is done 
to avoid confusion with interval notation, as we frequently restrict ourselves to the two dimensional case, writing 
$x(i,j)$ rather than the possibly confusing $x[i,j]$. 

\begin{definition}\label{def: complexity}
Suppose $X\subset\Sigma^{\Z^d}$ is a \zds, endowed with the natural $\Z^d$-action by translations. 
If $\CS\subset\Z^d$ is finite and $\alpha\colon\CS\to\Sigma$, define the {\em cylinder set} 
$$ 
[\CS,\alpha]:=\{\eta\in X\colon\text{the restriction of $\eta$ to $\CS$ is $\alpha$}\}. 
$$ 
The set of all cylinder sets forms a basis for the topology of $X$.  The {\em complexity function} for $X$ is the map $P_X\colon\{\text{finite subsets of $\Z^d$}\}\to\N$ given by 
$$ 
P_X(\CS):= \big \vert\{\alpha\in\Sigma^{\CS}\colon[\CS,\alpha]\neq\emptyset\} \big \vert
$$ 
which counts the number of colorings of $\CS$ which are
restrictions of elements of $X$.  If $\alpha\colon\CS\to\Sigma$,
is the restriction of an element of $X$ we say it 
{\em extends uniquely} to an $X$-coloring if there is exactly one legal $\eta \in X$ 
whose restriction to $\CS$ is $\alpha$.  Similarly, if $\CS\subset\CT\subset\Z^d$ and if $\alpha\colon\CS\to\Sigma$ is such that $[\alpha,\CS]\neq\emptyset$, then we say $\alpha$ {\em extends uniquely} to an $X$-coloring of $\CT$ if there is a unique $\beta\colon\CT\to\Sigma$ such that $[\beta,\CT]\neq\emptyset$ and the restriction of $\beta$ to $\CS$ is $\alpha$. 
\end{definition} 

Note that as in the one dimensional setting, the complexity function is translation invariant, 
meaning that for any $v \in \Z^d$, we have 
\[
P_X(\CS) = P_X(\CS + v).
\]

\subsection{Expansive subspaces}

An important concept in the study of higher dimensional systems
is the notion of an expansive subspace (see Boyle and Lind~\cite{BL} in
particular). For our purposes it suffices to restrict to the case $d = 2$.

\begin{defn}\label{def: expansive}
Suppose $X\subset \Sigma^{\Z^2}$ is a \z2s
and $L$ is a  one-dimensional subspace of  $\R^2$. We consider $\Z^2 \subset \R^2$
in the standard way.
For $r>0$, define 
$$L(r) = \{ z \in \Z^2 \colon d(z, L) \le r\}.$$
We say that the the line $L$ is {\em expansive} if
there exists $r>0$  such that for any $\eta \in X$, 
the restriction  $\eta|_{L(r)}$   extends uniquely to an $X$-coloring of $\Z^2$.  
We call the one-dimensional subspace $L$ {\em nonexpansive} 
if it fails to be expansive.
\end{defn}

It is also important for us to consider one-sided expansiveness for
a subspace $L$.  To 
define this we need to specify a
particular side of a one-dimensional subspace. For this 
we require an orientation of 
$\R^2$ (or $\Z^2)$ and an orientation of the subspace.  We use the standard
orientation of $\R^2$ given by the two form $\omega = dx \wedge dy$ 
or equivalently the orientation for which the standard ordered basis 
$\{(1,0),(0,1)\}$ is positively oriented.

If $L$ is an {\em oriented} one-dimensional subspace of $\R^2$ 
then the orientation determines a choice of one component
$L^+$ of  $L \setminus \{0\}$ which we call the {\em positive}
subset of $L$.
We then denote by $H^+(L)$ the open half space in 
$\R^2 \setminus L$ with the property that 
$\omega(v, w) > 0$ for all $v \in L^+$ and  $w \in H^+(L)$.  Alternatively,
$H^+(L)$ is the set of all $w \in \R^2$ such that 
$\{v,w\}$ is a positively oriented basis of $\R^2$ whenever
$v \in L^+$ and  $w \in H^+(L)$.  Equivalently
\[
H^+(L) = 
\{w \in \R^2 \colon i_v\omega(w) > 0\}
\]
whenever $v \in L^+$ and $i_v$ is the interior product.
The half space $H^-(L)$ is defined analogously or by $H^-(L) = -H^+(L)$.

\begin{defn}\label{def: sided-expansive}
Suppose  $L$ is an oriented one-dimensional subspace of $\R^2$, i.e.
it has a distinguished  choice of one component
$L^+$ of  $L \setminus \{0\}$ .
Then $L$ is {\em positively expansive} if there exists $r >0$ such that for
every $\eta \in X$, the restriction  $\eta|_{L(r)}$ extends uniquely
to the half space  $H^+(L)$.
Similarly $L$ is {\em negatively expansive} 
if the restriction  $\eta|_{L(r)}$ extends uniquely
to the half space  $H^-(L)$.
\end{defn}

\begin{prop}\label{prop: +expansive}
The oriented subspace $L$ is positively expansive if 
for every $\eta \in X$, the restriction  $\eta|_{H^-(L)}$ extends uniquely
to an $X$-coloring of $\Z^2$.  Equivalently 
$L$ fails to be positively expansive if and only
if there are colorings $\eta, \nu \in X$ such that $\eta \ne \nu$,
but $\eta(i,j) = \nu(i,j)$ for all $(i,j) \in H^-(L)$.
\end{prop}

\begin{proof} 
Suppose $L$ is positively expansive and $\eta,\nu\in X$ are such that $\eta(i,j) = \nu(i,j)$ for all $(i,j) \in H^-(L)$.  Find $r$ such that for any $\xi\in X$, $\xi|_{L(r)}$ extends uniquely to the half-space $H^+(L)$.  Let $v\in H^-(L)$ be such that the functions $\eta_v, \nu_v\in X$ defined by $\eta_v(x)=\eta(x+v)$ and $\nu_v(x)=\nu(x+v)$ have the same restriction to $L(r)\cup H^-(L)$.  Then by positive expansiveness of $L$, $\eta_v$ and $\nu_v$ coincide on $H^+(L)$ and hence on all of $\Z^2$.  So $\eta_v=\nu_v$ and it follows that $\eta=\nu$.  In other words, the restriction of $\eta$ to $H^-(L)$ extends uniquely to an $X$-coloring of $\Z^2$. 

Now suppose that for all $\eta\in X$ the restriction $\eta|_{H^-(L)}$ extends uniquely to an $X$-coloring of $\Z^2$.  We claim that $L$ is positively expansive.  For contradiction, suppose that for all $r>0$ there exist $\eta_r, \nu_r\in X$ such that $\eta_r|_{L(r)}=\nu_r|_{L(r)}$ but there exists 
$a_r\in H^+(L)$ such that $\eta_r(a_r)\neq\nu_r(a_r)$.  Define 
$$ 
B_r=\{(i,j)\in H^+(L)\colon\eta_r(i,j)\neq\nu_r(i,j)\}. 
$$ 
Let $H$ be the intersection of all closed half-planes (in $\R^2$) contained in $H^+(L)$
that contain $B_r$.  Fix some $x\in B_r$.  These half-planes are
linearly ordered by inclusion, all of them are contained in $H^+(L)$,
and all of them contain $x$. Thus their intersection is a closed half-plane
(which might not have any integer points on its boundary).  Therefore we can
find a closed half-plane $J\subseteq H^+(L)$, with integer points on its boundary, that contains $H$ and is such
that for all $y\in J\cap\Z^2$ there exists $z\in H\cap\Z^2$ with $\|y-z\|\leq1$.
Choose an integer vector $w_r\in\R^2\setminus J$ such that there exists
$v_r\in B_r\cap\Z^2$ satisfying $\|w_r-v_r\|\leq2$.  Finally, define
  $\eta_{r,w_r}, \nu_{r,w_r}\in X$ by $\eta_{r,w_r}(y)=\eta_r(y+w_r)$
  and $\nu_{r,w_r}(y)=\nu_r(y+w_r)$.  
Note that although vectors $w_r$ are not bounded,  we shift $\eta$ and $\nu$ so that $w_r$ 
is moved to the origin.  This shift is in the direction taking $H^-(L)$ into itself and thus preserves orientation in $\R^2$, 
ensuring that the shifted functions still agree on $H^-(L)$.  The purpose of the shift is that the point 
at which the functions disagree now can be bound in a bounded set.)
Then $\eta_{r,w_r}|_{H^-(L)}=\nu_{r,w_r}|_{H^-(L)}$ but there exists
$t_r\in H^+(L)\cap\left([-2,2]\times[-2,2]\right)$ such that
$\eta_{r,w_r}(t_r)\neq\nu_{r,w_r}(t_r)$.  We pass to a
subsequence $r_1<r_2<\cdots$ such that $t_r$ is constant.  By
compactness of $X$, we can pass if needed to a further subsequence along which
$\eta_{r_k,w_{r_k}}$ and $\nu_{r_k,w_{r_k}}$ both converge; call these
limiting functions $\eta_{\infty}$ and $\nu_{\infty}$.  By
construction $\eta_{\infty}(t_{r_1})\neq\nu_{\infty}(t_{r_1})$, but
$\eta_{\infty}|_{H^-(L)}=\nu_{\infty}|_{H^-(L)}$, a contradiction.
\end{proof}

\begin{prop}\label{prop: exp line}
Assume that $X\subset\Sigma^{\Z^2}$ is a \z2s and
 $L$ is a one-dimensional oriented subspace in the $u,v$-plane.  
Suppose there is a convex polygon $P \subset \R^2$ such that
\begin{enumerate}
\item There is a finite set $F \subset \Z^2$ such that $P$ is the  convex hull of 
$F$.
\item There is a unique $e \in F$ which is an extreme point of $P$
and which lies in $H^+(L)$.
\item For any $\eta \in X$, the restriction of $\eta$ to $F \setminus \{e\}$ 
extends uniquely to $F$.
\end{enumerate}
Then $L$ is positively expansive.  
\end{prop}
 
\begin{proof} 
For contradiction, suppose not.  Let $\eta, \nu\in X$ be such that $\eta|_{H^-(L)}=\nu|_{H^-(L)}$, but $\eta\neq\nu$.  Define $B=\{(i,j)\in H^+(L)\colon\eta(i,j)\neq\nu(i,j)\}$.  For each $b\in B$, define $d(b,L)$ to be the distance from $b$ to $L$ and let 
$$ 
I=\inf\{d(b,L)\colon b\in B\}. 
$$ 
For each $f\in F\setminus\{e\}$, let $d(f,e)$ be the distance between lines $L_e$ and $L_f$ parallel to $L$ that pass through $e$ and $f$, respectively.  Since $e\in H^+(L)$ and $f\notin H^+(L)$, for all $f\in F\setminus\{e\}$, we have $L_e\neq L_f$.  Thus 
$$ 
\varepsilon:=\min\{d(f,e)\colon f\in F\setminus\{e\}\}>0. 
$$ 

If there exists $b\in B$ such that $d(b,L)=I$, then define $\tilde{\eta}, \tilde{\nu}\in X$ by $\tilde{\eta}(x)=\eta(x+b-e)$ and $\tilde{\nu}(x)=\nu(x+b-e)$.  Then $\tilde{\eta}|_{H^-(L)}=\tilde{\nu}|_{H^-(L)}$ but $\tilde{\eta}(e)\neq\tilde{\nu}(e)$.  This contradicts the fact that the restriction of $\eta$ to $F\setminus\{e\}$ extends uniquely to an $X$-coloring of $F$. 

If for all $b\in B$ we have $d(b,L)>I$, then there exists $b\in B$ such that $d(b,L)-I<\varepsilon/2$.  Define $\tilde{\eta}, \tilde{\nu}\in X$ by $\tilde{\eta}(x)=\eta(x+b-e)$ and $\tilde{\nu}(x)=\nu(x+b-e)$.  Then $\tilde{\eta}(e)\neq\tilde{\nu}(e)$, but $\tilde{\eta}|_{F\setminus\{e\}}=\tilde{\nu}|_{F\setminus\{e\}}$, again a contradiction. 
\end{proof}

\begin{exs} \label{ex:three-dots}
~ 
\begin{enumerate}
\item Suppose $(X, \sigma)$ is a  shift and 
$\phi = \sigma^k,\  k \ne 0$. If $L$ is the line $i = k j$ 
and  $L^+ = L \cap \{(u,v) \colon v > 0\}$, then
$L$ is neither positively or negatively expansive,  but all other lines are expansive.

\item(Ledrappier's three dot system~\cite{ledrappier}). With the alphabet $\Sigma = \{0,1\}$, consider the subset of $\Sigma^{\Z^2}$ defined by 
$$
x(i,j) + x(i+1, j) + x(i, j+1) = 0 \pmod 2
$$
for all $i,j\in\Z$.  Other than the horizontal axis, the vertical
axis, and the reflected diagonal $y = -x$, every one-dimensional
subspace is expansive.  None of these three subspaces is expansive,
but each of them is either positively or negatively expansive.
\item(Algebraic examples; see~\cite{BL,ELMW} for further background).
With the alphabet $\Sigma = \{0,1\}$, consider the subset of $\Sigma^{\Z^2}$ defined by 
\[
x(i,j) + x(i+1, j+1) + x(i-1, j+2) = 0 \pmod 2
\]
for all $i,j\in\Z$.  It is not difficult to see that the subspaces parallel 
to the sides of the triangle with vertices $(0,0), (1,1)$, and
$(-1,2)$ each fail to be one of positively or negatively expansive (but not both).
All other one-dimensional subspaces are expansive.
\end{enumerate}
\end{exs}

\section{The spacetime of an endomorphism}\label{sec: spacetime}

\subsection{$\phi$-coding}
We continue to assume that $(X, \sigma)$ is an infinite shift over the finite alphabet $\Sigma$.

Some of the results in this section overlap with results of Nasu~\cite{nasu}, where he studies endomorphisms of subshifts that are resolving, which roughly speaking is a notion of being determined. While his language and terminology are different from ours, Lemma~\ref{lemma: welldef}  and Proposition~\ref{theta-cont} correspond to results in Section 6 of~\cite{nasu} and the limiting objects given in Definition~\ref{alpha def} and some of their properties (portions of Proposition~\ref{alpha prop}) are described in Section 9 of~\cite{nasu}.

\begin{defn}
\label{def:phi-codes}
If $\phi \in \End(X, \sigma)$ is an endomorphism
we say that a subset $A \subset \Z$ {\em $\phi$-codes} (or simply {\em codes}
if $\phi$ is clear from context)
a subset $B \subset \Z$ if for any $x,y \in X$ satisfying $x[a] = y[a]$ for all $a \in A$, it follows
that $\phi(x)[b] = \phi(y)[b]$ for all $b \in B$.
\end{defn}

We remark that if $\phi\in\End(X,\sigma)$ is an endomorphism then, as $\phi$ is determined by a block code of some range (say $R$), the ray $(-\infty,0]$ $\phi$-codes the ray $(-\infty,-R]$.  Similarly the ray $[0,\infty)$ $\phi$-codes the ray $[R,\infty)$.  Of course, it could be the case that $(-\infty,0]$ $\phi$-codes a larger ray than $(-\infty,-R]$.  This motivates the following definition: 

\begin{defn}
If $\phi \in \End(X, \sigma)$ and $n \ge 0$, 
let $W^+(n, \phi)$ be the smallest element of $\Z$ such
that the ray $[W^+(n, \phi), \infty)$ is $\phi^n$-coded by $[0,\infty)$
meaning that if $x$ and $y$ agree on
$[0,\infty)$, then necessarily $\phi^n(x)$ and $\phi^n(y)$ agree
on $[W^+(n, \phi), \infty)$ and this is the largest ray with that property.
Similarly $W^-(n,\phi)$ is the largest element of  $\Z$ such
that the ray $(-\infty ,W^-(n, \phi)]$ is $\phi^n$-coded by $(-\infty, 0]$.
When $\phi$ is clear from the context, we omit it from the notation and  denote $W^+(n, \phi)$ and $W^-(n, \phi)$ by 
$W^+(n)$ and $W^-(n)$, respectively.  
\end{defn}

Note that for $n\ge 1$ we have  $W^+(n, \phi) = W^+(1, \phi^n)$
and $W^-(n, \phi) = W^-(1, \phi^n)$.  These quantities have been studied in~\cite{sherevesky, tisseur} in order to define Lyapunov exponents for cellular automata, and then used to study the speed of propagation of perturbations with respect to a shift invariant measure.  They use
this to give bounds on the entropy of the measure in terms of these (left and right) Lyapunov exponents.  We do not consider the role of an invariant measure in this article, but we give an estimate
for topological entropy closely related to a result of ~\cite{tisseur} (see our Theorem~\ref{thm: entropy} below).



We check that $W^+(n, \phi)$ and $W^-(n,\phi)$ are well-defined: 
\begin{lemma}\label{lemma: welldef}  
If $X$ is infinite, then $W^+(n,\phi)>-\infty$ and $W^-(n,\phi)<\infty$. 
\end{lemma} 
\begin{proof} 




For contradiction, suppose $W^+(n,\phi)=-\infty$ so that whenever $x,y\in X$ and $x[0,\infty)=y[0,\infty)$ we have $\phi x=\phi y$.  Let $R$ denote the range of the block code defining $\phi$. 

 For any fixed $s>0$, we claim that there exists $M\in\N$
such that if $x,y\in X$ and $x[0,M]=y[0,M]$, then $(\phi x)[-s,M-R]=y[-s,M-R]$.  For 
contradiction, suppose not.  Then there exist sequences $(x_n)$ and
$(y_n)$ of points in $X$ such that $x_n[0,n]=y_n[0,n]$, but
$(\phi x_n)[-s,n-R]\neq(\phi y_n)[-s,n-R]$.  Since $\phi$ is a block code, observe that $(\phi x_n)[R,n-R]=(\phi y_n)[R,n-R]$.  Passing to a subsequence if necessary, 
we can assume that $x_{n+1}[0,n]=x_n[0,n]$, $(\phi x_{n+1})[-s,n-R]=(\phi x_n)[-s,n-R]$, $y_{n+1}[0,n]=y_n[0,n]$, and $(\phi y_{n+1})[-s,n-R]=(\phi y_n)[-s,n-R]$ for all $n\in\N$.  Let $z_x, z_y\in X$ be such
that $z_x[0,n]=x_n[0,n]$, $(\phi z_x)[-s,n-R]=(\phi x_n)[-s,n-R]$, $z_y[0,n]=y_n[0,n]$, and $(\phi z_y)[-s,n-R]=(\phi y_n)[-s,n-R]$ for all $n\in\N$.
Then $z_x[0,\infty)=z_y[0,\infty)$ but $(\phi z_x)[-s,\infty)\neq(\phi z_y)[-s,\infty)$,  a
contradiction.  This proves the claim and shows that the integer $M$ exists.  

Taking $s=R+1$, observe that if $w\in\mathcal{L}_{M+1}(X)$ then there exists $u(w)\in\mathcal{L}_{M+2}(X)$ such that for any $x\in X$ and for any $i\in\Z$ such that $w=x[i,i+|w|-1]$, we have $u=(\phi x)[i-s,i+|w|]$.  Since $\phi$ is surjective, for any $u\in\mathcal{L}_{M+2}(X)$ there exists $w\in\mathcal{L}_{M+1}(X)$ such that $u=u(w)$.  In other words, $P_X(M+2)\leq P_X(M+1)$.  But $P_X$ is nondecreasing and so $P_X(M+2)=P_X(M+1)$.  It follows inductively that $P_X(M+k)=P_X(M+1)$ for any $k\in\N$.  But then $X$ is finite,  a contradiction.  Therefore $W^+(n,\phi)>-\infty$. 

The argument that $W^-(1,\phi)<\infty$ is similar. 
\color{black}
\end{proof} 

By Lemma~\ref{lemma: welldef}, the function 
$\Theta^+_n\colon \Sigma^{[0,\infty)} \to \Sigma^{[W^+(n,\phi),\infty)}$ defined by
\[
\Theta^+_n(x[0,\infty)) = \phi^n(x)[W^+(n,\phi),\infty)
\]
is well defined for all $n\ge 0$, as is the analogous function
$\Theta^-_n\colon  \Sigma^{(-\infty, 0]} \to \Sigma^{(-\infty,W^-(n,\phi)]}$. These functions are continuous: 
\begin{prop}\label{theta-cont}
The functions $\Theta_n^+$ and $\Theta_n^-$ are continuous.  In particular, 
there exists $k = k(n, \phi) >0$ such that $[0,k]$ $\phi^n$-codes $\{W^+(n,\phi)\}$
and $[-k,0]$ $\phi^n$-codes $\{W^-(n,\phi)\}$.
\end{prop}

\begin{proof}
Assume $\Theta_n^+$ is not continuous. Then there exist $x_j$ and
$y$ in $X$ and $r \ge W^+(n,\phi) $ such that
$x_j[0, m_j] = y[0,m_j]$, for a sequence $\{m_j\}$ with $\displaystyle\lim_{j\to\infty} m_j = \infty$, and such that $\phi(x_n)[r] \ne \phi(y)[r]$.  By passing to a subsequence, w can assume that 
there exists $z \in X$ with $\displaystyle \lim_{n\to\infty} x_n = z$.  Clearly
$z[0, \infty) = y[0, \infty)$ and hence
$\phi^n(z)[W^+(n,\phi), \infty) = \phi(y)[W^+(n,\phi), \infty)$.  In
particular,  $\phi(z)[r] = \phi(y)[r]$, and so by continuity of $\phi$ we
conclude that $\displaystyle\lim_{n\to\infty} \phi(x_n)[r] = \phi(z)[r] = \phi(y)[r]$.  But
since $\phi(x_n)[r] \ne \phi(y)[r]$,  we also have that 
$\displaystyle\lim_{n\to\infty} \phi(x_n)[r] \ne \phi(y)[r]$, a contradiction.  
Thus $\Theta_n^+$ is continuous, and a similar argument shows that 
$\Theta_n^-$ is continuous.
\end{proof}

\subsection{The spacetime of $\phi$}
\begin{defn}
If $\phi \in \End(X, \sigma)$ is an endomorphism, 
its {\em $\phi$-spacetime} $\U  = \U(\phi)$ is a \z2s
together with a preferred ordered
basis for $\Z^2$ which defines what we call the 
``horizontal''  and ``vertical'' directions. It
is defined to be the closed subset of  $x \in \Sigma^{\Z^2}$ such that
for all $i \in \Z$ and $j\ge 0$
$\phi^j(x)[i] = x(i, j)$.

\end{defn}
Thus the rows of $\U$ are elements of $X$ with row $n$ equal 
to $\phi$ of row $n-1$.  There is an action of $\Z^2$ on $\U$
given by having $(i,j)$ shift $i$ times in the horizontal direction
and $j$ times in the vertical direction.  A vertical shift by $j \ge 0$ can
also be viewed as applying $\phi^j$ to each row of $\U$.

It follows immediately from the definition of expansiveness (Definition~\ref{def: expansive})  
that the horizontal axis in a spacetime $\mathcal{U}$ of an automorphism
is always an expansive 
subspace for the \z2s $\mathcal{U}$ with the $\Z^2$-action by translations. 
Also if $L$ is the horizontal axis in  the spacetime of an
endomorphism  and $L^+$ is the intersection of $L$ with the positive
horizontal axis,  then $H^+(L)$ is the upper half space
and $L$  is positively expansive.

Note that given a spacetime $\U$ (including the preferred basis of $\Z^2$), one can extract
the shift $(X, \sigma)$ by taking  $X$ to be the $\Sigma$-colorings of $\Z$ obtained by
restricting the colorings in $\U$ to the $i$-axis ($j= 0$).  Likewise,  one can 
extract the endomorphism $\phi$ by using the fact that if $y \in \U$ and
$x \in X$ is given by $x[i] = y[i, 0]$, then $\phi(x)[i] = y[i, 1]$.

A concept somewhat more general than our notion of spacetime
is defined in Milnor~\cite{milnor} and referred to as the  {\em complete history}
of a cellular automaton. Our context is narrower, using the spacetime to study a 
single endomorphism rather than the full system. However, there are analogs in our 
development; Milnor defines an $m$-step forward cone, 
which corresponds to our interval $[W^-(m,\phi), W^+(m, \phi)]$, his definition of a 
limiting forward cone corresponds to our asymptotic light cone, and the case $n_0=0$ of Theorem~\ref{thm: interval-code}
corresponds to results in Milnor.  

We say that spacetimes $\U$ and $\U'$, which share the same 
alphabet $\Sigma$, are {\em spacetime isomorphic} if there is a homeomorphism
$h\colon \U \to \U'$ such that 
\[
h(z)(i',j') = z(i,j), 
\]
where the isomorphism of $\Z^2$ for which $(i,j) \mapsto (i',j')$ is given
by sending the preferred basis of $\Z^2$ for $\U$ to the preferred basis of $\U'$. 
(Note that the assumption that the spacetimes share the same alphabet is not necessary, but 
simplifies our notation.)
It is straightforward to check  that 
$\phi, \phi' \in \Aut(X) $ are conjugate automorphisms
(see Definition~\ref{def: conjugacy}) if and only if
their respective spacetimes 
are spacetime isomorphic.

We extend definition~\ref{def:phi-codes} of coding to a spacetime: 
\begin{defn}
If $\U$ is a \z2s, we say that a subset 
$A \subset \Z^2$ {\em codes} a subset $B \subset \Z^2$ if 
for any $x,y \in \U$ satisfying $x(i,j) = y(i,j)$ for all $(i,j) \in A$, it follows
that $x(i',j') = y(i',j')$ for all $(i', j') \in B$. Equivalently if $x$ and $y$
differ at some point of $B$,  they also differ at some point of $A$.
\end{defn}

\begin{defn}[{\bf Light Cone}]
The {\em future light cone $\C_f(\phi)$ of $\phi \in \End(X)$} is defined to be 
\[
\C_f(\phi) = \{(i,j) \in \Z^2 \colon W^-(j, \phi) \le i \le W^+(j, \phi),\  j \ge 0\}
\]
The {\em past light cone $\C_p(\phi)$ of $\phi$} is defined to be  
$\C_p(\phi) = - \C_f(\phi)$.  The {\em full light cone}
$\C(\phi)$ is defined to be $\C_f(\phi) \cup C_p(\phi)$.
 \end{defn}

The rationale for this terminology is that if $x \in X$ 
and $j >0$, then a 
change in the value of $x(0)$ (and no other changes)
 can only cause a change in $\phi^j(x)[i],\  j \ge 0$ if $(i,j)$ lies
in the future light cone of $\phi$.  Similarly if $\phi^j(y) = x,\ j \ge 0,$ 
then a change in $y[i]$  can only affect $x[0]$ if $(i, -j)$ lies
in the past light cone of $\phi$.

The light cone is naturally stratified into levels: 
define the {\em $n^{th}$ level of $\C(\phi)$} to be the set 
\begin{equation}\label{def:I}
\cI(n, \phi):=  \{ i \in \Z \colon (i, n) \in \C(\phi)\}.
\end{equation}

In Corollary~\ref{cor: I(n,phi)} below, we show that if $\sigma$ is
a subshift of finite type and $n$ is large, then the horizontal
interval in the light cone at level $-n$ i.e., 
$\cI(-n, \phi)$, is the unique minimal interval which
$\phi^{n}$-codes $\{0\}$,  provided $\phi$ has infinite order in
$\End(X,\sigma)/\langle\sigma\rangle$. 

In general, it is not clear if $\phi \in \Aut(X),$ what 
the relationship, if any, between $\C(\phi)$ and $\C(\phi^{-1})$ is.  
However there are some restrictions given in Part~\eqref{item:five} of Proposition~\ref{alpha prop}.

\begin{remark}
A comment about notation is appropriate here. We are interested in subsets
of the $i,j$-plane. Our convention is that $i$ is the abscissa, or first coordinate, 
and we consider the $i$-axis to be horizontal.  Likewise 
$j$ is the ordinate, or second coordinate, 
and we consider the $j$-axis to be vertical. However some subsets
of the plane we consider are naturally described as graphs of a function
$i = f(j)$.  For example, we frequently consider lines given by an
equation like $i = \alpha j,\ j \in \R$,  and think of $\alpha$ as a ``slope'' even
in standard parlance it would be the reciprocal of the slope of the
line $i = \alpha j$.
\end{remark}

Our next goal is to study the asymptotic behavior of 
$W^+(j, \phi)$ and $W^-(j, \phi)$ for a fixed $\phi \in \End(X)$.
We start by recalling Fekete's Lemma,  which is then applied to  the sequence 
$W^+(n) = W^+(n,\phi)$ for $n \ge 0$ which is shown to be subadditive.

\begin{lemma}[Fekete's Lemma~\cite{Fek}]\label{fekete}
If the sequence $a_n \in \R,\  n \in \N$, is subadditive (meaning that $a_n + a_m \ge a_{m+n}$ for all $m,n\in\N$), then
\[
\lim_{n \to \infty} \frac{a_n}{n} = \inf_{n \ge 1}\frac{a_n}{n}.
\]
\end{lemma}

We note a simple, but useful, consequence of this: 
if $s(n)$ is subadditive,  
and if $\displaystyle \lim_{n \to \infty} \frac{s(n)}{n} \ge 0$,  then $s(n) \ge 0$ for all $n \ge 1$
as otherwise $\displaystyle \inf_{m \ge 1}\frac{s(m)}{m}$ would be negative.

\begin{lemma}
\label{lemma:subadd}
If $\phi, \psi \in \End(X,\sigma)$ then
$W^+(1,\phi \psi) \le W^+(1,\phi) +W^+(1,\psi)$  and
similarly $W^-(1,\phi \psi) \ge W^-(1,\phi) + W^-(1,\psi)$.
In particular the sequences  $\{W^+(n, \phi)\}$ and
$\{-W^-(n,\phi)\}, \ n \ge 0$,  are subadditive.
\end{lemma}

\begin{proof}
The ray $[0,\infty)$ $\psi$-codes $[W^+(1,\psi),\infty)$ 
and the ray $[W^+(1,\psi),\infty)$ $\phi$-codes $[W^+(1,\phi) 
+ W^+(1,\psi),\infty)$. Hence $[0, \infty)$ $\phi \psi$-codes  
$[W^+(1,\phi) +W^+(1,\psi),\infty)$
so $W^+(1,\phi \psi) \le W^+(1,\phi) + W^+(1,\psi)$. This proves the
first assertion.

Replacing $\phi$ by $\phi^m$ and $\psi$ by $\phi^n$ in this
inequality gives
\[
W^+(1, \phi^{n+m}) \le W^+(1, \phi^{m}) + W^+(1, \phi^{n}).
\]
Since for $n\ge 1$ we have  $W^+(n, \phi) = W^+(1, \phi^n)$
we conclude that $ W^+(m+n, \phi) \le W^+(m, \phi) + W^+(n, \phi)$,
so $\{W^+(n, \phi)\}$ is subadditive. The proof for $W^-$ is similar.
\end{proof}

We now want to consider two quantities which measure the asymptotic
behavior of $W^\pm(n, \phi)$.  These quantities (and other closely related
ones) have been considered in ~\cite{sherevesky, tisseur} in the context
of measure preserving cellular automata and are referred to there as
Lyapunov exponents of the automaton.
If we fix $\phi$ and abbreviate $W^+(n, \phi)$ by
$W^+(n)$ then Fekete's Lemma and Lemma~\ref{lemma:subadd}, 
imply the limit  $\displaystyle{ \lim_{n \to \infty} \frac{W^+(n)}{n}}$ exists.

\begin{defn} \label{alpha def}
We define
\[
\alpha^+(\phi) := \lim_{n \to \infty} \frac{W^+(n)}{n}
\] and 
\[
\alpha^-(\phi) := \lim_{n \to \infty} \frac{W^-(n)}{n}. 
\]
\end{defn}

Note that the limit $\alpha^+(\phi)$ is finite, since 
if $D \ge \rr(\phi)$, then for $j \ge 0$ we have
$|W^+(j))| \le D j$ (and $|W^-(j))| \le D j$). 
As a consequence, we conclude that 
\begin{equation}\label{eq:W-plus}
W^+(n) = n \alpha^+(\phi) + \lo(n).  
\end{equation}
This describes an important asymptotic property of the right light cone boundary function
$W^+(n)$ used in the proof of Theorem~\ref{thm nonexpansive} below.  That theorem
says that if $\alpha^+ = \alpha^+(\phi)$, then
the  line $i = \alpha^+ j$ 
is a nonexpansive subspace of  $\R^2$ for the spacetime of $\phi.$ 

Similarly, we can consider $W^-(n)$ and obtain a second nonexpansive subspace
namely the line $x = \beta y$ where
\[
\beta = \alpha^-(\phi) := \lim_{n \to \infty} \frac{W^-(n)}{n}. 
\]
As a consequence, we conclude that 
the left light cone boundary function satisfies
\begin{equation}
\label{eq:W-minus}
W^-(n) =n \alpha^-(\phi)+\lo(n).
\end{equation}

We list some elementary properties of the limits $\alpha^{+}(\phi)$ and $\alpha^{-}(\phi)$: 
\begin{prop}\label{alpha prop}
If $\phi \in \End(X, \sigma)$ then
\begin{enumerate}
\item
\label{item:one}
 For all $k \in \Z,\ \alpha^-(\sigma^k\phi) 
= \alpha^-(\phi) +k$ and $\alpha^+(\sigma^k\phi) 
= \alpha^+(\phi) +k$.
\item 
\label{item:two}
For all $m \in \N, \alpha^+(\phi^m) = m \alpha^+(\phi)$
and $\alpha^-(\phi^m) = m \alpha^-(\phi)$
\item 
\label{item:three}
If $X$ is infinite, then $\alpha^-(\phi) \le \alpha^+(\phi)$
\item 
\label{item:four}
If $\phi, \psi \in \Aut(X, \sigma)$ are commuting endomorphisms
then 
\[
\alpha^+(\phi \psi) \le \alpha^+(\phi) + \alpha^+(\psi) \text{ and }
\alpha^-(\phi \psi) \ge \alpha^-(\phi) + \alpha^-(\psi).
\]
\item \label{item:five}
If $\phi$ is an automorphism and X is infinite, then 
\[
\alpha^+(\phi) + \alpha^+(\phi^{-1}) \ge 0 \text{ and }
\alpha^-(\phi) + \alpha^-(\phi^{-1}) \le 0.
\]

\end{enumerate}
\end{prop}

\begin{proof}
Since 
$$W^+(n, (\sigma^k\phi)) = W^+(1,\sigma^{nk}\phi^n) = W^+(1, \phi^n) +nk
= W^+(n, \phi) +nk,$$ 
property~\eqref{item:one} follows. Since
\[
\lim_{n \to \infty}\frac{W^+(mn, \phi)}{n}
= m \lim_{n \to \infty} \frac{W^+(mn, \phi)}{mn} = m \alpha^+(\phi), 
\]
property~\eqref{item:two} follows.

To show~\eqref{item:three},  we observe that it
suffices to show $W^+(n, \phi) \ge W^-(n, \phi)$ for all $n >0$.
For the purpose of contradiction we assume that there exists $n>0$
with  $W^-(n) > W^+(n)$.
By Proposition~\ref{theta-cont}, $\Theta_n^+$ is continuous and so 
there exists $R>0$ such that the interval $[0,R]$ $\phi^n$-codes the entry at $W^+(n)$.
Therefore for all $t \ge 0$ the interval $[0,R+t]$ $\phi^n$-codes the interval
$[W^+(n),W^+(n)+t]$. Clearly $R$ could be replaced by any larger value and
this still holds.
Similarly, there exists $R^{\prime}>0$ such
that the interval $[-R^{\prime},0]$ $\phi^n$-codes the entry at $W^-(n)$ and so by translating, 
$[0, R']$\ $\phi^n$-codes the entry $W^-(n) + R^{\prime}$. Just as with $R$, the value 
$R^{\prime}$ can be replaced by any larger value and hence we can assume that $R = R^{\prime}$.
Then the interval $[0,R]$ $\phi^n$-codes the entry at  $W^-(n) + R$ and therefore 
for $t\ge 0$, the interval $[0,R+t]$ $\phi^n$-codes 
the interval $[W^-(n) + R, W^-(n) + R +t]$.

Thus for $t\ge 0$, the interval $[0,R+t]$ $\phi^n$-codes both the interval
$[W^+(n),W^+(n)+t]$ and the interval $[W^-(n)+R, W^-(n)+R+t]$.
Increasing $R$ if necessary, we can assume that $W^-(n)+R \ge W^+(n)$.
Note that for any $t > W^-(n)-W^+(n)+R$, we have $W^-(n) + R < W^+(n) +t$.  Thus 
the two intervals $[W^+(n),W^+(n)+t]$ and $[W^-(n)+R, W^-(n)+R+t]$
overlap and their union is the interval $[W^+(n), W^-(n)+R+t]$. 
Thus for any sufficiently large $t$, the interval $[0,R+t]$, which
has length $R+t+1$,\ $\phi^n$-codes
the interval $[W^+(n), W^-(n)+R+t]$ with length
$W^-(n)-W^+(n)+R+t+1$, which is greater than $R+t+1$.  

This implies that 
$$P_X(R+t+1)\geq P_X(R+t+1+W^-(n)-W^+(n)),$$ 
since every word of length
$R+t+1+W^-(n)-W^+(n)$ is $\phi^n$-determined by some word of length
$R+t+1$.  Since $P_X$ is a nondecreasing function, we have
\begin{equation}\label{P_X eqn}
P_X(R+t+1)=P_X(R+t+1+W^-(n)-W^+(n))
\end{equation}
for all sufficiently large $t$.  Choose $t_0$ large enough such that
this equation holds when $t = t_0$ and define $t_k = t_{k-1} +  W^-(n)-W^+(n)$ for
$k \ge 1$.  Then by Equation~\eqref{P_X eqn}
\[
P_X(R+t_{k-1}+1) = P_X(R+t_{k-1}+1+W^-(n)-W^+(n))
= P_X(R+t_{k}+1).
\]
So by induction on $k$
\[
P_X(R+t_{k}+1) = P_X(R+t_{0}+1).
\]
Therefore the function $P_X$ is bounded above by $P_X(R+t_{0}+1)$.
It follows that for any $m >0$, there are at most 
$P_X(R+t_{0}+1)$ allowable colorings of the interval $[-m, m]$.
This contradicts our standing assumption that $X$ is infinite
and establishes~\eqref{item:three}.

To prove~\eqref{item:four} we note that
\begin{align*}
W^+(n, \phi \psi) &= W^+(1, (\phi \psi)^n)
= W^+(1, \phi^n \psi^n) \\
&\le W^+(1, \phi^n ) + W^+(1, \psi^n)  = W^+(n, \phi ) + W^+(n, \psi).
\end{align*}
Hence,
\[
\lim_{n \to \infty} \frac{W^+(n, \phi \psi)}{n} \le
\lim_{n \to \infty} \frac{W^+(n, \phi)}{n} + 
\lim_{n \to \infty} \frac{W^+(n, \psi)}{n}
\]
giving the inequality of item~\eqref{item:four}. The result for $\alpha^-$
is similar.

Item ~\eqref{item:five}  follows immediately from ~\eqref{item:four} if we replace
$\psi$ with $\phi^{-1}$, since $\alpha^+(id)  = \alpha^-(id) = 0$.
\end{proof}

Other than the restriction that $\alpha^-(\phi)\leq \alpha^+(\phi)$, any rational values can be taken on for some automorphism of the full shift: 
\begin{ex}\label{ex: cone}
We show that given rationals $r_1 \le r_2$, there is a full shift $(X,\sigma)$ with an automorphism $\phi$
such that $\alpha^-(\phi) = r_1$ and $\alpha^+(\phi) = r_2$. 

Suppose that $r_2 = p_2/q_2 \ge 0$.
Consider $X_2$ the Cartesian product of $q_2$ copies of the full two shift 
$\sigma\colon  \{0,1\}^\Z \to \{0,1\}^\Z$.
Define an automorphism $\phi_0$ by having it cyclically permuting the copies of
$\{0,1\}^\Z$ and perform a shift on one of them.  Then $\phi_0^{q_2} = 
\sigma_2\colon X_2 \to X_2$ is the shift (indeed a full shift on an alphabet of size $2^{q_2}$).  
Since $\alpha^+(\sigma_2) = \alpha^-(\sigma_2) = 1$, 
it follows from parts~\eqref{item:one} and~\eqref{item:two} of Proposition~\ref{alpha prop}
that $\alpha^+(\phi_0) = \alpha^-(\phi_0) = 1/q_2$. Setting $\phi_2 = \phi_0^{p_2}$, we have that 
$\alpha^+(\phi_2) = p_2\alpha^+(\phi_2) = p_2/q_2 =r_2$.  Similarly
$\alpha^-(\phi_2) = r_2$. 
If $r_2 = -p_2/q_2 < 0$  we
can do the same construction, defining $\phi_0$ to cyclically permute the copies of
$\Sigma_2$ but use the inverse shift  (instead of the shift) on one of the copies.  
Then $\phi^{q_2} = \sigma_2^{-1}\colon X_2 \to X_2$. In this way we still construct $\phi_2$
with $\alpha^+(\phi_2) = \alpha^-(\phi_2) = r_2$.  

By the same argument we can construct an
automorphism $\phi_1$ of $(X_1, \sigma_1)$ such that $\alpha^+(\phi_1)
= \alpha^-(\phi_1) = r_1$.  
Taking $X$ to be the Cartesian product $X_1 \times X_2$
and considering the (full) shift $\sigma = \sigma_1 \times \sigma_2\colon  X \to X$, 
and the automorphism $\phi = \phi_1 \times \phi_2$, it is straightforward
to check that $\alpha^+(\phi) =  \alpha^+(\phi_2) = r_2$ and
$\alpha^-(\phi) =  \alpha^-(\phi_1) = r_1$.
\end{ex}

In light of the work  on Lyapunov exponents for cellular automata, it is natural to ask 
for a general shift $\sigma$  and $\phi \in Aut(X, \sigma)$ which conditions on $\phi$ and/or $\sigma$
 suffice for the existence of a $\sigma$-invariant $\phi$-ergodic measure $\mu$
such that $\alpha^\pm(\phi)$  are Lyapunov exponents in the sense defined by~\cite{sherevesky, tisseur}.

\subsection{Two dimensional coding}

\begin{remark} 
We thank Samuel Petite for suggesting the short proof of the following lemma (in an earlier version of this paper we had a longer proof of this lemma). 
\end{remark} 

\begin{lemma} \label{lem: bounded range}
Let $\varphi\in\End(X)$ and suppose that there exists $K$ such that $\rr(\varphi^n)\leq K$ for infinitely many $n$.  Then $\varphi$ has finite order. 
\end{lemma} 
\begin{proof} 
There are only finitely many block maps of range $\leq K$ and so, by the pigeonhole principle, there exist $0<m<n$ such that $\varphi^m=\varphi^n$.  It follows that $\varphi^{n-m}$ is the identity. 
\end{proof}

Recall that the interval $\cI(n, \phi)$  is defined in equation~\eqref{def:I}
to be $\{ i \in \Z \colon (i, n) \in \C(\phi)\}$.
Thus for $n\in\N$, we have 
$|\cI(-n, \phi)|=W^+(n, \phi)-W^-(n,\phi) +1$, is the width of the
  $n^{th}$ level of the light cone for $\phi$.

\begin{lemma}\label{interval-code}
Suppose $\phi$ is an endomorphism of the shift $(X,\sigma)$ and $n \ge 0$.
If  $J$ is any interval in $\Z$ which  $\phi^n$-codes $\{0\}$,
then $J \supset \cI({-n, \phi})  $.
\end{lemma}

\begin{proof}
If the interval $J = [a,b]$ $\phi^n$-codes $\{0\}$, 
then $[a,\infty)$
 $\phi^n$-codes $[0, \infty)$,  and so $[0,\infty)$ 
$\phi^n$-codes $[-a, \infty)$.  It follows that
$-a \ge W^+(n, \phi)$ and hence $a \le -W^+(n, \phi)$.
Similarly $b \ge -W^-(n, \phi)$ and so $\cI(-n, \phi) \subset [a,b]$.
\end{proof}

\begin{lemma}
\label{lem: interval-code2}
Assume $\phi$ is an endomorphism
of a shift of finite type $(X,\sigma)$ and suppose that 
\[
\lim_{n \to \infty} \big |\cI(-n, {\phi}) \big | = +\infty.
\]
Then there is
$n_0$ such that whenever $n \ge n_0$, the interval
$\cI(-n, \phi)$\  $\phi^n$-codes $\{0\}$. Moreover,  
if $\sigma$ is a full shift we can take $n_0$ to be $0$ and
the hypothesis $\displaystyle \lim_{n \to \infty}  |\cI(-n, {\phi}) | = \infty$ is unnecessary.
\end{lemma}

In slightly more generality,  if $(X, \sigma)$ is a $1$-step shift of finite type, then we can take $n_0$ to be $0$.
 
\begin{proof}  
Suppose  that $\phi$ is an endomorphism and $\sigma$ is a subshift of finite
type.  Then by Proposition~\ref{bowen prop}, 
there exists $m_0\geq 0$ such that if $w$ is a word of length at least $m_0$ 
and if $w_1^- w w_1^+$ and $w_2^- w w_2^+$ are  elements of $X$ for
some semi-infinite words $w_i^\pm$, then both
$w_1^- w w_2^+$ and $w_2^- w w_1^+$ are  elements of $X$.
Clearly $m_0 =0$ suffices if $\sigma$ is a full shift.

By hypothesis, 
\[
\lim_{n \to \infty} \big | \cI(-n, \phi) \big |
= \lim_{n \to \infty} \big |W^+(n, \phi) -W^-(n, \phi) \big | +1  = +\infty,
\]
and so we can choose $n_0$ such
that the length of $\cI(-n, \phi)$ is greater than $m_0$ when $n \ge n_0$.
Suppose  $n \ge n_0$ and that $x,y \in X$ agree 
on the interval $\cI(-n, \phi)$.  
We show that $\phi^n(x)[0] = \phi^n(y)[0]$.  
Let $w = x[-W^+(n,\phi), -W^-(n,\phi)] =
y[-W^+(n,\phi), -W^-(n,\phi)]$ and define $w_i^\pm$ by
$x(-\infty , \infty) = w_1^- w w_1^+$ and
$y(-\infty , \infty) = w_2^- w w_2^+$.  Then $w_1^- w w_2^+$ 
is an element of $X$ satisfying $x[i] =  w_1^- w w_2^+[i]$  
for all $i \le -W^-(n,\phi)$ and
$y[i] =  w_1^- w w_2^+[i]$   for all $i \ge -W^+(n,\phi)$.  It follows that
$\phi^n(x)[0] = \phi^n(w_1^- w w_2^+)[0]$
and that $\phi^n(y)[0] = \phi^n(w_1^- w w_2^+)[0]$.
Hence $\phi^n(x)[0] = \phi^n(y)[0]$ and $\{0\}$ is $\phi^n$-coded by $[-W^+(n,\phi), -W^-(n,\phi)].$

\end{proof}
\begin{definition} 
  Let $X$ be a subshift and let $\phi\in\End(X,\sigma)$.  Define
  $r(n, \phi)$ to be the minimal width of an interval which
  $\phi^n$-codes $\{0\}$.
\end{definition} 

\begin{lemma}\label{bound} 
Suppose $(X, \sigma)$ is a subshift of finite type and $\phi \in \End(X,\sigma)$.
Then there is a constant
$C(\phi)$ such that
$|\cI(-n, \phi)|\leq r(n, \phi) \leq |\cI(-n, \phi)|+C(\phi)$.  If $X$ is
a full shift we can take $C(\phi) = 0$.
\end{lemma}
 
In slightly more generality,  if $(X, \sigma)$ is a $k$-step shift of finite type, then $C(\phi)$ can be taken to be $k-1$.  
 
\begin{proof} 
The first inequality follows immediately from Lemma~\ref{interval-code}.  We
prove the second inequality by contradiction.  Thus suppose that 
for any $C$,  there exist infinitely many
$n\in\N$ and points $x_{C,n}\neq y_{C,n}$ which agree on the interval
$[-W^+(n, \phi),-W^-(n,\phi)+C]$ but are such that
$\phi^n(x_{C,n})[0]\neq\phi^n(y_{C,n})[0]$.

Recall from Proposition~\ref{bowen prop} that there exists a constant
$n_0$ (depending on the subshift $X$) such that if $x,y\in X$ agree
for $n_0$ consecutive places, say $x[i]=y[i]$ for all $p\leq i<p+n_0$, then
the $\Z$-coloring whose restriction to $(-\infty,p+n_0-1]$ coincides
with that of $x$ and whose restriction to $[p+n_0,\infty)$ coincides
with that of $y$, is an element of $X$. 

Choose $C > n_0$.
By assumption, there exist infinitely many $n\in\N$ and points
$x_{n}, y_{n}\in X$ which agree on
$[-W^+(n, \phi),-W^-(n, \phi)+C]$ but are such that
$\phi^n(x_{n})[0]\neq\phi^n(y_{n})[0]$. Let $z\in X$ be the
$\Z$-coloring whose restriction to $(-\infty,-W^-(n, \phi)+C]$
coincides with $x_{n}$ and whose restriction to
$[-W^-(n, \phi)+C+1,\infty)$ coincides with $y_{n}$.  
Then since $C> n_0$, we have that $z\in X$. 
 Since $z$ agrees with $y_{n}$ on
$[-W^+(n, \phi),\infty)$, it follows that $(\phi^nz)[0]=(\phi^ny_{n})[0]$.  
On the other hand,
$(\phi^nz)[0]=(\phi^nx_{n})[0]$,  since $z$ agrees with $x_{n}$ on
$(-\infty,-W^-(n,\phi)]$.  But this contradicts the fact that
$(\phi^nx_{n})[0]\neq(\phi^ny_{n})[0]$, and so $C(\phi)$ exists.
\end{proof} 

\begin{proposition} \label{prop: finite order}
Suppose $X$ is a subshift of finite type and $\phi\in\End(X,\sigma)$.  If 
$$ 
\liminf_{n\to\infty}|\cI(-n, \phi)|<\infty, 
$$ 
then $\phi$ has finite order in $\End(X,\sigma)/\langle\sigma\rangle$. 
\end{proposition} 

\begin{proof} 
By hypothesis, there exists $M$ such that $|\cI(-n,\phi)|<M$ for infinitely many $n$.  By Lemma~\ref{bound}, there is a constant $C(\phi)$ such that $r(n, \phi)<M+C(\phi)$ for infinitely many $n$.  Let $n_1<n_2<\cdots$ be a subsequence along which $r(n_i,\phi)<M+C(\phi)$ is constant; define this constant to be $R$.  Then for each $i=1,2,\dots$ there is an interval $[a_i,b_i]$ of length $R$ which $\phi^{n_i}$-codes $\{0\}$.  Therefore the interval $[0,R]$ $(\sigma^{-a_i}\phi^{n_i})$-codes $\{0\}$ for all $i$.  It follows that $\sigma^{-a_i}\phi^{n_i}$ is a block map of range $R$ for all $i$.  There are only finitely many block maps of range $R$, so there must exist $i_1<i_2$ such that $\sigma^{-a_{i_1}}\phi^{n_{i_1}}=\sigma^{-a_{i_2}}\phi^{n_{i_2}}$ or simply 
$$ 
\phi^{n_{i_1}}(x)=\sigma^{a_{i_1}-a_{i_2}}\phi^{n_{i_2}}(x)=\sigma^{a_{i_1}-a_{i_2}}\phi^{n_{i_2}-n_{i_1}}\left(\phi^{n_{i_1}}(x)\right) 
$$ 
for all $x\in X$.  Since $\phi^{n_{i_1}}$ is a surjection, we have 
$$ 
y=\sigma^{a_{i_1}-a_{i_2}}\phi^{n_{i_2}-n_{i_1}}(y) 
$$ 
for all $y\in X$.  In other words, $\phi^{n_{i_2}-n_{i_1}}=\sigma^{a_{i_2}-a_{i_1}}$. 
\end{proof} 

\begin{thm}
\label{thm: interval-code}
Assume that $\phi$ is an endomorphism
of a shift of finite type $(X,\sigma)$ and that 
$\phi$ has infinite order in $\End(X)/\langle \sigma \rangle.$
Then there exists $n_0$ such that whenever $n \ge n_0$, the interval
$\cI(-n, \phi)$  $\phi^n$-codes $\{0\}$. 
If $\sigma$ is a full shift, we can take $n_0$ to be $0$.
\end{thm}

\begin{proof}
If $\sigma$ is a full shift, the result follows from Lemma~\ref{lem: interval-code2}.
Otherwise, since $\phi$
has infinite order in $\End(X)/\langle \sigma \rangle$, Proposition~\ref{prop: finite order}
tells us that
\[
\lim_{n\to\infty}|\cI(-n,\phi)| = +\infty.  
\]
Thus we can apply Lemma~\ref{lem: interval-code2} to conclude 
that $\cI(-n, \phi)$  $\phi^n$-codes $\{0\}$. 
\end{proof}

\begin{cor}\label{cor: I(n,phi)}
If $\phi$ has infinite order in $\End(X)/\langle \sigma \rangle$, then for $n$
sufficiently large, $\cI(-n, \phi)$ is the unique minimal interval
which $\phi^n$-codes $\{0\}$. 
\end{cor}

\begin{proof} The fact that $\cI(-n, \phi)$\ $\phi^n$-codes $\{0\}$ for 
large $n$ follows from Theorem~\ref{thm: interval-code}.  Minimality
and uniqueness follow from Lemma~\ref{interval-code}.
\end{proof}

\begin{question} 
Is the hypothesis that $(X,\sigma)$ is an SFT necessary in Theorem~\ref{thm: interval-code}? 
\end{question} 

\section{The light cone and nonexpansive subspaces}
\label{sec:lightcone}

The main result of this section is Theorem~\ref{thm nonexpansive}: it states 
that the line $u = \alpha^+(\phi) v$ in the $u,v$-plane
is a nonexpansive subspace of $\R^2$ for the spacetime of $\phi$.  
The analogous statement holds in the other direction: 
the line $u = \alpha^-(\phi) v$ in the $u,v$-plane
is a nonexpansive subspace.

\subsection{The deviation function} 

We begin by investigating the properties of the function which measures
the deviation of $W^+(n,\phi)$ from $\alpha^+(\phi) n$.

\begin{definition}
Suppose $\phi \in \End(X, \sigma)$.
For $n \ge 0$ define the
positive and negative {\em deviation functions} $\delta^+(n) = \delta^+(n,\phi)$
and $\delta^-(n) = \delta^-(n,\phi)$
by $\delta^+(n) = W^+(n) - n \alpha^+(\phi)$
and $\delta^-(n) = W^-(n) - n \alpha^-(\phi)$
\end{definition}

\begin{lemma}\label{deviation lemma}
Suppose $\delta^+(n)$ and $\delta^-(n)$  are the deviation functions
associated to $\phi$.
Then 
\begin{enumerate}
\item
\label{it-dev:partone}
The functions $\delta^+(n)$ and $-\delta^-(n)$ are subadditive.
\item 
\label{it-dev:parttwo}
The deviation functions satisfy $\displaystyle \lim_{n\to \infty} \frac{\delta^+(n)}{n} = 0$ and
$\displaystyle \lim_{n\to \infty} \frac{\delta^-(n)}{n} = 0$.
\item 
\label{it-dev:partthree}
For all $n\geq 0$, we have $\delta^+(n) \ge 0$ and
$\delta^-(n) \le 0$.
\end{enumerate}
\end{lemma}

\begin{proof}
Since $\delta^+(n)$ is the sum of the subadditive function
$W^+(n) = W^+(\phi^n)$ and the additive function $-n \alpha$, 
part~\eqref{it-dev:partone} follows.  
Since
\[
\lim_{n\to \infty} \frac{\delta^+(n)}{n} = \lim_{n\to \infty} \frac{W^+(n) -n \alpha^+(\phi)}{n}
= \lim_{n\to \infty} \frac{W^+(n)}{n} - \alpha^+(\phi) =0, 
\]
part~\eqref{it-dev:parttwo} follows.

To see part~\eqref{it-dev:partthree},  observe that parts~\eqref{it-dev:partone} and~\eqref{it-dev:parttwo} together with Fekete's 
Lemma (Lemma~\ref{fekete}) imply
\[
\inf_{n \ge 1} \frac{\delta^+(n)}{n} = 0
\]
and so $\delta^+(n) < 0$ is impossible. The analogous results
for $\delta^-(n)$ are proved similarly.
\end{proof}

\begin{lemma}\label{spacetime lemma}
Let $\U$ be the $\phi$-spacetime of $(X,\sigma)$ for $\phi \in \End(X)$ and 
let $\alpha = \alpha^+(\phi)$ and $\delta(n) = \delta^+(n,\phi)$.
Suppose that $\alpha \ge 0$ and the deviation $\delta(n)$ is unbounded 
for $ n \ge 0$.
Then there exist two sequences $\{x_m\}_{m \ge 1}$ and $\{y_m\}_{m \ge 1}$  in $\U$
such that 
\begin{enumerate}
\item
\label{item:one1}
$x_m(i,j) =   y_m(i,j)$ for all $(i,j)$ with $-m \le j \le 0$ and $i \ge \alpha j$
\item
\label{item:two2}
$x_m(i,j) =   y_m(i,j)$ for all $(i,j)$ with $j\ge 0 $ and $i \ge (\alpha + \frac{1}{m}) j$
\item
\label{item:three3}
$x_m(-1,0) \ne y_m(-1,0)$ for all $m \in \N$.  
\end{enumerate}
The analogous result for $\alpha^-(\phi)$ and $\delta^-(n,\phi)$ also holds.
\end{lemma}

\begin{proof}
For notational simplicity, denote  $W^+(n)$ by $W(n)$,
so $\delta(n)  = W(n) -n \alpha$.

We define a piecewise linear $F(t)$ from the set
$\{t\in\Z \colon  t \ge -m\}$ to $\Z$ and show $W(t) \le F(t)$ for all $t\geq -m$.
We then use this to define $x_m, y_m$ satisfying the three properties. 

Given $m \in \N$ and using the facts that
$\displaystyle \lim_{k \to \infty} \frac{W(k)}{k} = \alpha$
and $\displaystyle \lim_{k \to \infty} \frac{\delta(k)}{k} = 0$, 
we can  choose $n_0 = n_0(m) > m$ such that 
\[
\frac{\delta(k)}{k} <  \frac{1}{m}
\] for all $k >n_0$.
For the moment as $m$ is fixed we suppress the dependence of $n_0$ on $m$.
By hypothesis, $\delta(k)$ is unbounded 
above and so we can also choose 
$n_0$ so that
\begin{equation} \label{dn_0 > dj}
\delta(n_0) > \delta(j) \text{ for all } 0 \le j < n_0.
\end{equation}

Define a  line $i = L(j)$ in the $i,j$-plane by
\[
L(j)= \frac{1}{m} (j - n_0) + \delta(n_0).
\]

We claim that the set of $j$ with $\delta(j) \ge L(j)$ is finite.
By Lemma~\ref{deviation lemma},
\[
\lim_{j \to \infty} \frac{\delta(j)}{j -n_0} = \lim_{j \to \infty} \frac{\delta(j)}{j} = 0
\]
and  so  for sufficiently large $j$,
\[
\delta(j) \le \frac{1}{m} (j- n_0) <  \frac{1}{m}(j- n_0)+ \delta(n_0) = L(j),
\]
since $\delta(n_0) \ge 0$ (by Lemma~\ref{deviation lemma}). This proves the claim. 

Let  $J$ be the finite set $\{j\ \colon\  \delta(j)  \ge L(j),\ j\ge 0\}$ and
let $S = \{(\delta(j), j)\ \colon\ j \in J\}$.
Note that $S \ne \emptyset$ since $(\delta(n_0), n_0) \in S$.

Let $t_0 = t_0(m) \in  \N$ be the value of $j$ with $j \ge n_0$ 
for which $\delta(j) - L(j)$ is maximal.
Then $(\delta(t_0), t_0) \in S$.  Since, for the moment $m$ is fixed,  we suppress
the $m$ and simply write $t_0$ for $t_0(m)$.

Suppose now that $j \in [n_0,t_0]$.  
Then since $\delta(t_0) -L(t_0) \ge \delta(j) -L(j)$, it follows 
that   $\delta(t_0) \ge \delta(j) + L(t_0) - L(j) \ge \delta(j)$
since $j \in [n_0,t_0]$  and $L$ is monotonic increasing.
Thus  we have
\begin{equation}\label{eq:dt ge dj}
\delta(t_0) \ge \delta(j) \text{ for all } j \in [n_0,t_0].
\end{equation}

Let $\displaystyle \am = \alpha + \frac{1}{m}$ and consider the two lines
\[
i = \CK (j), \text{ where } \CK(j) = \alpha(j- t_0) + W(t_0)
\]  
and
\[
i = \CL (j), \text{ where } \CL(j) = \am(j- t_0) + W(t_0).
\]  
Both lines pass through $(W(t_0), t_0)$.

Define
\begin{equation}\label{def Fj}
F(j) =
\begin{cases}
\CK(j), &\text{if $0 \le j \le t_0$}\\
\CL(j), &\text{if $j \ge  t_0$}.
\end{cases}
\end{equation} 
We claim that for all $j \ge 0$
\begin{equation*} 
W(j) \le F(j).
\end{equation*} 

We prove this claim by considering  two separate
ranges of values for $j$, first $j \ge t_0$, then
$0 \le j \le t_0$.

In the range $j \ge t_0 $, by the choice of $t_0$ 
 we have  that $\delta(j) - L(j) \le \delta(t_0) - L(t_0)$ if $j \in J$.
But the same inequality holds for $j \notin J$ since then $\delta(j) - L(j) <0$ 
and $\delta(t_0) - L(t_0) \ge 0$.  Thus 
$\delta(j) \le L(j) +\delta(t_0) - L(t_0)$ for all $j \ge t_0$.  Therefore
\begin{align*}
W(j) & =  \delta(j) +  \alpha j\\ 
& \le   L(j) +\delta(t_0) - L(t_0) + \alpha j\\
& = \frac{1}{m} (j -n_0)  - \frac{1}{m} (t_0 -n_0) + \delta(t_0) + \alpha j\\
& = \frac{1}{m} (j -t_0) + \delta(t_0) + \alpha t_0 + \alpha(j -t_0)\\
& = \alpha_m (j -t_0) + \delta(t_0) + \alpha t_0 \\
&= \alpha_m (j -t_0) + W(t_0)   \\ 
& =   \CL(j).
\end{align*}
This proves the claim for the first range, i.e.,
\begin{equation}
\label{t0 le j} 
W(j)\leq\CL(j)\text{ for $j \ge t_0$.}
\end{equation}

Next we consider the range $0 \le j \le t_0$.
Note if $j \le n_0$ then $W(j)  = \delta(j) +  \alpha j 
\le  \delta(n_0) + \alpha j$ by Equation~\eqref{dn_0 > dj}, so
$W(j) \le  \delta(t_0) + \alpha j$  since $\delta(t_0) \ge \delta(n_0)$
by Equation~\eqref{eq:dt ge dj}.
But if $j \in [n_0,t_0]$ then
$W(j)  = \delta(j) +  \alpha j 
\le  \delta(t_0) + \alpha j$  by Equation~\eqref{eq:dt ge dj}.
So we conclude $W(j)  \le  \delta(t_0) + \alpha j$ for all $0 \le j \le t_0$.

Hence in this range
\begin{align*} 
W(j) &\le  \delta(t_0) + \alpha j \\
&=   \delta(t_0) +\alpha t_0 + \alpha (j -t_0)\\ 
& =  W(t_0) + \alpha (j -t_0) =   \CK(j).  
\end{align*}
Thus we have
\begin{equation} \label{0 le j le t0}
W(j) \le \CK(j) \text{ for $ 0 \le j  \le t_0$}
\end{equation}

Hence Equations~\eqref{t0 le j}, and~\eqref{0 le j le t0}  establish the claim, 
demonstrating that 
\begin{equation} \label{Fj ge Wj}
F(j) \ge W(j) \text{ for all } j \ge 0,
\end{equation} 
where
\begin{equation*}
F(j) =
\begin{cases}
\CK(j) &\text{if $0 \le j \le t_0$}\\
\CL(j) &\text{if $j \ge  t_0$}.
\end{cases}
\end{equation*} 

We now use this to define the elements $x_m$ and $y_m$.  From 
the definition of $W^+(n,\phi)$ (which we are denoting $W(n)$),  we know that 
whenever $j \ge 0$ and  $u,v \in X$ have the rays 
$u[0,\infty)$ and $v[0,\infty)$ equal,
it follows that the rays $\phi^j(u)[ W(j),\infty) = \phi^j(v)[ W(j),\infty)$.
Equivalently if $x$ and $y$ are  the  elements in
$\phi$-spacetime which  agree on
the ray $\{ (i,0) \in \Z^2 \colon i \ge 0\}$, then 
\begin{equation}\label{i > Wj}
j \ge 0 \text{ and } i \ge W(j) \text{ implies } x(i,j) = y(i,j).
\end{equation}
Moreover for each  $j\ge 0$, there exist $u_j, v_j \in X$ such that 
$u_j[0,\infty) =  v_j[0,\infty)$ but
\[
\phi^j(u_j)( W(j) -1) \ne  \phi^j(v_j)( W(j)-1).  
\]

In particular  this means that for $m \in \N$ there exist elements $\hat x_m,  \hat y_m\in \U$ 
which are equal on the ray $\{ (i,0) \in \Z^2 \colon i \ge 0\}$, but such that 
\begin{equation}\label{W minus 1}
\hat x_m( W(t_0(m)) -1, t_0(m)) \ne \hat y_m( W(t_0(m))-1,t_0(m)).
\end{equation}
(Note that the dependence of $t_0 = t_0(m)$ on $m$ is now salient so we return
to the more cumbersome notation.)
We use translates of $\hat x_m$ and $\hat y_m$ by the vectors
$(W(t_0(m)), t_0(m)) = (\delta(t_0(m)) + \alpha t_0(m), t_0(m))$ to define
$x_m, y_m \in \U$. More precisely, define 
$$x_m(i,j) = \hat x_m(i +W(t_0(m)), j+ t_0(m))$$  and
$$y_m(i,j) = \hat y_m(i +W(t_0(m)), j+ t_0(m)).$$  
Note that $x_m$ and $y_m$ agree on the ray $\{ (i,0) \in \Z^2 \colon i \ge 0\}$.

We proceed to check properties~\eqref{item:one1},~\eqref{item:two2}, and~\eqref{item:three3}
of the lemma's conclusion.

From Equation~\eqref{W minus 1} and the definition of $x_m$ and $y_m$ we have
$$x_m(-1,0) = \hat x_m(W(t_0(m))-1, t_0(m)) \ne \hat y_m(W(t_0(m))-1, t_0(m)) = y_m(-1,0),$$
and so~\eqref{item:three3} follows. 

To check~\eqref{item:one1}, suppose $-m \le j \le 0$ and $i\ge \alpha j$.  
Let $i' = i + W(t_0)$ and
$j' = j+ t_0$ and so $x_m(i, j) = \hat x_m(i', j')$ and $y_m(i, j) = \hat y_m(i', j')$. 
Hence if we show that $\hat x_m(i',j') = \hat y_m(i',j')$, 
then we have that $x_m(i,j) = y_m(i,j)$, which is the statement of~\eqref{item:one1}. 
This in turn follows from Equation~\eqref{i > Wj}
if we show $j' \ge 0$ and $i' \ge W(j')$.  We proceed to do so.  

Note that since $-m \le j \le 0$  and since, by construction, $n_0(m)$ and $t_0(m)$ 
satisfy  $m < n_0(m)  < t_0(m)$), we have
\[
0 \le t_0(m) -m \le t_0(m) +j  = j'
\]
To show $i' \ge W(j')$ observe that 
since $i \ge j \alpha$, it follows that 
$i' = i + W(t_0)  \ge j \alpha  + W(t_0)  = 
 (j' -t_0) \alpha  + W(t_0) = \CK(j')$.
Since $j' = j+ t_0(m) \le t_0(m)$ the definition of $F$ (equation \ref{def Fj}) shows
$\CK(j') = F(j')$, and 
we may apply Equation~\eqref{Fj ge Wj} to conclude
$i' \ge W(j')$.  Then by Equation~\ref{i > Wj} applied to
$\hat x_m$ and $\hat y_m$ at $(i',j')$ we have
$\hat x_m(i',j') = \hat y_m(i',j')$, so $x_m(i,j) = y_m(i,j)$.
This completes the proof of property~\eqref{item:one1}.

To check~\eqref{item:two2}, we assume that $j \ge 0$ and $i \ge \am j$. 
Again we let  $i' = i + W(t_0(m))$ and $j' = j+ t_0(m)$ and so $j' \ge t_0(m)$.  To
show $x_m(i,j) = y_m(i,j)$, 
it suffices to show $\hat x_m(i',j') = \hat y_m(i',j')$ when
\[
 j' \ge t_0(m) \text{  and } i' \ge \am j + W(t_0(m)) 
\]
But $ \am j + W(t_0(m)) =  \am (j' -t_0) +  W(t_0) = \CL(j')$,
so we have $ j' \ge t_0(m)$  and  $i' \ge \CL(j')$.

Since $j' \ge t_0(m)$ we conclude from the definition of 
$F$ (Equation~\eqref{def Fj}) that $F(j') = \CL(j')$.
So $i' \ge F(j') $ and hence by Equation~\eqref{Fj ge Wj}, 
$i' \ge F(j')  \ge W(j')$. Since $i' \ge W(j')$
we have $x_m(i',j') = y_m(i',j')$ by Equation~\ref{i > Wj}, completing the proof of~\eqref{item:two2}.  

The proof of the analogous result for $\alpha^-(\phi)$ and $\delta^-(n,\phi)$ is done
similarly.
\end{proof}

\subsection{Nonexpansiveness of light cone edges}

\begin{thm}\label{thm nonexpansive}
Suppose $\phi \in \End(X,\sigma)$  and $\alpha^+ = \alpha^+(\phi)$.
In the spacetime $\U$ 
of $\phi$ orient the line $u = \alpha^+ v$ so
that $\la \alpha^+, 1\ra$ is positive.
Then this oriented line is not a positively expansive subspace.
Similarly  if  $\alpha^- = \alpha^-(\phi)$, 
the line $u = \alpha^- v$ (oriented so that $\la \alpha^-, 1\ra$ is
positive) is  not a negatively expansive subspace.
\end{thm}

\begin{proof}
Let $\U$ be the $\phi$-spacetime of $(X,\sigma)$.
Replacing $\phi$ with $\sigma^k \phi^m$ and using 
part~(1) of Proposition~\ref{alpha prop}, 
without loss of generality we can assume that $\alpha^+(\phi)  \ge 0$.

{\bf Case 1: bounded deviation.}
As a first case we assume that the non-negative deviation function 
$\delta$ is bounded. Say $\delta(j) <  D$ for some $D > 0$ and all $j \in \N$. Since $\delta(j) \ge 0$
and $\alpha^+ \ge 0$, we have  $0 \le W^+(j, \phi) -\alpha^+ j = \delta(j) < D.$ 

If we have two elements $x, y \in \U$ satisfying $x(k,0) =  y(k,0)$ for $k \ge 0$
then whenever $j \ge 0$ and $i \ge D + \alpha^+ j$, we have
$i > W^+(j)$.  Hence
\begin{equation}\label{bounded case}
x(i, j) = y(i, j) \text{ for all } j \ge 0 \text{ and } i \ge D + \alpha^+ j
\end{equation} (see Equation~\eqref{i > Wj}).
Thus $x$ and $y$ agree in
the part of the upper half space to the right of the line 
$i = D + \alpha^+ j$. 

By the definition of $W^+(n) = W^+(n, \phi)$ for $n \in \N$
we may choose $\hat x_n, \hat y_n \in \U$  which agree
on the ray $\{ (i,0) \in \Z^2 \colon  i \ge 0\}$ such that
 $\hat x_n(W^+(n) -1, n) \ne \hat y_n(W^+(n) -1, n)$.

We want to create new colorings by translating $\hat x_n$ and $\hat y_n$ by
the vector $(W^+(n),n)$.  
More precisely for $n \ge 0$ we define $ x_n$ and $ y_n$
by $x_n(i,j) = \hat x_n(i + W^+(n), j + n)$.
Note that $x_n(-1, 0) \ne y_n(-1, 0)$, since
$x_n(-1, 0) = \hat x_n(W^+(n) -1, n) \ne \hat y_n(W^+(n) -1, n) = y_n(-1, 0)$.

For all $j \ge -n$ and $i \ge  D + \alpha^+ j$,  we claim that 
\[
x_n(i, j) = y_n(i, j). 
\]
To see this define $i' = i + W^+(n)$ and  $j' = j+n$ and so
$x_n(i, j) = \hat x_n(i', j')$ and $y_n(i, j) = \hat y_n(i', j')$.
Then 
\begin{align*}
i' &= i + W^+(n) \\
&\ge  D +  \alpha^+ j+ W^+(n) \\
&= D + \alpha^+ j'  +(W^+(n) - \alpha^+ n)\\
&= D + \alpha^+ j'  +\delta(n)\\
& \ge D + \alpha^+ j'.
\end{align*}

Hence $\hat x_n(i',j')$
and $\hat y_n(i',j')$ are equal by Equation~\ref{bounded case}
whenever $i \ge D + \alpha^+ j$ and $j \ge -n$ (since $j' \ge 0$ when $j \ge -n$).
But $x_n(i,j) = \hat x_n(i',j')$ and $y_n(i,j) = \hat y_n(i',j')$ so 
$x_n(i,j) = y_n(i,j)$.
Thus $x_n$ and $y_n$ agree at $(i,j)$ if
$i \ge D + \alpha^+ j$ and $j \ge -n$. 

Since $\U$ is compact
we can choose convergent subsequences (also denoted $x_n$ and $y_n$).
Say $\lim x_n = \hat x$ and $\lim y_n = \hat y$.  Then clearly
$\hat x(-1, 0) \ne \hat y(-1, 0)$ and $\hat x(i, j) = \hat y(i, j)$
for all $i > D  + \alpha^+ j$. So $\hat x$ and $\hat y$ agree on the
half space $H^+ = \{(i,j)\colon  i > D + \alpha^+ j\}$.  This implies  the oriented
line $ u = \alpha^+ v$ is not positively expansive.  The case of the line $u = \alpha^- v$ is handled
similarly.

{\bf Case 2: unbounded deviation.} 
We consider  the elements $x_m, y_m$ guaranteed by 
Lemma~\ref{spacetime lemma}, and recall that they satisfy properties~\eqref{item:one1}-\eqref{item:three3} 
of the lemma.  

Since $\U$ is compact, by passing to subsequences, we can assume 
that both sequences converge in $\U$, say to
$\hat x$ and $\hat y$.  Clearly $\hat x(-1,0) \ne \hat y(-1,0)$.  
We claim the colorings  $\hat x$ and $\hat y$ agree on the half space 
$H^+ = \{ (i,j)\colon  i > \alpha j\}$  of  $\Z^2$.  
It then follows that the oriented line $u = \alpha v$ is not positively expansive
(see Definition~\ref{def: expansive}).

To prove the claim, note that if $(i,j) \in H^+$,\  
$-m_0 \le j \le 0$ and $m \ge m_0$ then $x_m(i,j) = y_m(i,j)$.
Hence the limits satisfy $\hat x(i,j) = \hat y(i,j)$ 
whenever $(i,j) \in H^+$ and $j \le 0$.  But also if
$j > 0$ and $i > \alpha j$, then for some $n_0 > 0$
we have $i \ge (\alpha + \frac{1}{n_0}) j$ and it follows that
$x_m(i,j) =   y_m(i,j)$ whenever $m > n_0$.  Hence the limits satisfy
$\hat x(i,j) = \hat y(i,j)$.   

The case of the line $u = \alpha^- v$ is handled
similarly.
\end{proof}

\subsection{Expansive subspaces}
We want  to investigate which one-dimensional 
subspaces in a spacetime are expansive.
Since the horizontal axis in a spacetime is always positively expansive
for an endomorphism and expansive for an automorphisms, we restrict our
attention to lines in $\R^2 = \{(u,v)\}$ given by $u = mv$ where $m \in \R$.
(We write the abscissa as a function of the ordinate for convenient comparison
with the edges of $\A(\phi)$ which are $u = \alpha^+ v$ and $u = \alpha^- v$.)

\begin{prop}\label{prop m>alpha}
Suppose $L$ is a line  in $\R^2$ given by  $u = mv$  and
oriented so that $\la m, 1\ra$ is positive.  Then: 
\begin{enumerate}
\item
\label{item:1} 
If $m > \alpha^+(\phi)$, then  $L$ is positively expansive.
\item
\label{item:2} 
If $m < \alpha^-(\phi)$, then  $L$ is negatively expansive.
\end{enumerate}
Moreover if $\phi$ is an automorphism and if $m > \max\{ \alpha^+(\phi), -\alpha^-(\phi^{-1})\}$ 
or if $m < \min\{ \alpha^-(\phi), -\alpha^+(\phi^{-1})\}$, 
then $L$ is expansive.
\end{prop}
\begin{proof}
We first consider~\eqref{item:1}.  We show that if
$\U$ is the spacetime of $\phi$ and $x, y \in \U$ agree 
on the right side of $u = mv$, then they also agree on the left side.
This implies that the oriented line $L$ is positively expansive.
Since $m > \alpha^+(\phi)$, the vector
$\langle \alpha^+(\phi),1 \rangle$ is not parallel to $L$ and
points in the direction from the right side of $L$ to the left side.

Let $W^+(n) = W^+(n,\phi)$ so
\[
\lim_{n \to \infty} \frac{W^+(n)}{n} = \alpha^+(\phi)
\]
(see Equation~\eqref{alpha def}) and hence
\[
\lim_{n \to \infty} \frac{1}{n}\langle W^+(n),n \rangle = \langle \alpha^+(\phi), 1 \rangle.
\]

It follows that for sufficiently large $n$, the vector $\langle W^+(n), n \rangle$
is also  not parallel to $L$ and 
points in the direction from the right side of $L$ to the left side.
Hence, given any $(u_0, v_0) \in \Z^2$ on the left side of $L$, there exists $n_0 >0$ 
such that if $u_1 = u_0 - W^+(n_0)$ and $v_1 = v_0 -n_0$, 
then $(u_1 , v_1)$ is on the right side of $L$.  The ray
$\{ (t, v_1) \colon u_1 \le t \}$ in $\U$ lies entirely to the right 
of $L$ and codes $\{(u_0,v_0)\}$.  

It follows that if $x, y \in \U$
agree to the right of $L$, then they also agree at $(u_0,v_0).$  Since
$(u_0,v_0)$ is an arbitrary point to the left of $L$, it follows that
$L$ is positively expansive.   The proof of~\eqref{item:2} is analogous.

To show the final statement, note the the reflection
$R\colon \R^2 \to  \R^2 $ given by $R(u,v) = (u, -v)$ has the 
property that it switches the spacetimes $\U(\phi)$ and 
$\U(\phi^{-1})$, i.e., it induces a map $R^* \colon  \U(\phi) \to \U(\phi^{-1})$
given by $R^*(\eta) = \eta \circ R$.

If $L$ is the line $i = m j$, then our convention for the orientation of $L$ 
was chosen so that  
\[
L^+ = \{ \la u, v \ra \colon \la u, v \ra \in L \text{ and } v >0\}.
\]
 Hence the convention implies that $R(L^+)$ is the set of
{\em negative} vectors in $R(L)$  and
 the positive vectors in $R(L)$ are $R(L^-)$  where $L^- = -L^+$.
Note that $H^+(L)$ consists of the vectors above the line $L$ 
so $R(H^+(L))$ is the set of vectors below $R(L)$.  
(see Definition~\ref{def: sided-expansive} and the paragraph preceding it).
But since
$R$ reverses the orientation of $L$ we have $H^+(R(L)) = R(H^+(L)$.
It follows that $L$ is positively (resp. negatively) expansive in $\U(\phi)$ if and only
if $R(L)$ is positively (resp. negatively) expansive in $\U(\phi^{-1})$, i.e. $R$ acting
on non-vertical lines preserves
positive expansiveness and negative expansiveness.

Now consider the line $L$ given by $i = mj$ in $\U(\phi)$, and so $R(L)$
is the line $i = -mj$ in $\U(\phi^{-1})$.  By part~\eqref{item:2}, if
$-m < \alpha^-(\phi^{-1})$ (or equivalently if $m > -\alpha^-(\phi^{-1})$), 
then the line $R(L)$ is negatively expansive in
$\U(\phi^{-1})$. 
Hence
$m > -\alpha^-(\phi^{-1})$ implies that $L$ is negatively expansive
in $\U(\phi)$.  If we also have $m > \alpha^+(\phi)$ then by part~\eqref{item:1}, 
the line $L$ is also positively expansive and thus it is, in fact, 
expansive.  The case that $m < \min\{ \alpha^-(\phi), -\alpha^+(\phi^{-1})\}$ 
is handled similarly.
\end{proof}

\section{Asymptotic behavior}
\label{sec:asymptotic}

\subsection{The asymptotic light cone}
The edges of the light cone $\C(\phi)$ are given
by the graphs of the functions $i = W^+(j, \phi)$ 
$i = W^-(j,\phi)$.  Since these functions have nice
asymptotic properties, so does the cone they determine,
which motivates the following definition: 
\begin{defn} The {\em asymptotic light cone} of $\phi$ 
is defined to be 
\[
\A(\phi) = \{(u,v) \in \R^2 \colon \alpha^-(\phi) v \le  u \le \alpha^+(\phi)v \}. 
\]
\end{defn}

This means $\A(\phi)$ is the cone in $\R^2$ which does not
contain the $i$-axis and which is bounded by the lines
$u = \alpha^+(\phi) v$ and $u = \alpha^-(\phi) v$.
We view $\A(\phi)$ as a subset of $\R^2$ rather
than of $\Z^2$, as we want to consider lines with
irrational slope that may lie in $\A(\phi)$ but would intersect
$\C(\phi)$ only in $\{0\}$.

We begin by investigating 
the deviation of the function $W^+(n, \phi)$ from the linear
function $n \alpha^+(\phi)$.  Observe 
that the asymptotic light cone $\A(\phi)$ is a subset of the light
cone $\C(\phi)$, as an immediate corollary of part~\eqref{item:three} of
Lemma~\ref{deviation lemma}.

\begin{cor}
The set of integer points in the asymptotic light cone $\A(\phi)$ is a subset of the 
light cone $\C(\phi)$.
\end{cor}

If $\phi \in \Aut(X)$ it is natural to consider the relationship between 
$\C(\phi)$ and $\C(\phi^{-1})$, or between
$\A(\phi)$ and $\A(\phi^{-1})$.  The spacetime
$\U(\phi)$ of $\phi$ is not the same as the spacetime 
$\U(\phi^{-1})$ of $\phi^{-1}$, but there is a natural
identification of $\U(\phi)$ with the reflection of
$\U(\phi^{-1})$ about the horizontal axis $j = 0$. 
In general, it is not true 
that  $\A(\phi^{-1})$ is the reflection
of $\A(\phi)$ about the $u$-axis (Example~\ref{ex:three-dots} is one where this fails).  
On the other hand, if $(X, \sigma)$ is a subshift, there  is at least one line in the intersection of
 $\A(\phi^{-1})$ with the reflection of $\A(\phi)$ about the $u$-axis.

To see this, note that the cone $\A(\phi^{-1})$ has edges which are the lines
\begin{equation}
u = \alpha^+(\phi^{-1}) v \text{ and } u = \alpha^-(\phi^{-1}) v, 
\end{equation}
while the cone obtained by reflecting $\A(\phi)$ about the $u$-axis
has edges given by
\begin{equation}\label{eqn: cone edge}
u = -\alpha^-(\phi) v \text{ and } u = -\alpha^+(\phi) v. 
\end{equation}

Hence the line $u = mv$ lies in the intersection 
 $\A(\phi^{-1})$ and the reflection of $\A(\phi)$ in the line $u$-axis if
\[
m \in [\alpha^-(\phi^{-1}),\alpha^+(\phi^{-1})] \cap
[-\alpha^+(\phi),-\alpha^-(\phi)].
\]
If these two intervals are disjoint, then either
\[
\alpha^+(\phi^{-1}) < -\alpha^+(\phi) \text{ or }
-\alpha^-(\phi) < \alpha^-(\phi^{-1}).
\]
Either of these inequalities contradict part~\eqref{item:five} of Proposition~\ref{alpha prop}.

In a different vein, the cone $\A(\phi)$ is a conjugacy invariant: 
\begin{prop}\label{conj invariant}
Suppose $(X_i, \sigma_i)$ is a shift for $i = 1,2$ and 
$\phi_i \in \End(X_i)$. Suppose further that $h\colon X_1 \to X_2$ is a 
topological conjugacy from $\sigma_1$ to $\sigma_2$.
If
\[
\phi_2 = h \circ \phi_1 \circ h^{-1},
\]
then $\A(\phi_1) = \A(\phi_2)$.
\end{prop}

\begin{proof}
Since $h$ is a block code, there is a constant
$D >0$, depending only on $h$, such that 
for any $n \in \Z$ the ray
$[n,\infty)$ $h$-codes $[n+ D,\infty)$ and the ray $(-\infty, n]$ 
$h$-codes $(-\infty, n-D]$.  It follows that 
$W^+(m, \phi_1) \le W^+(m, \phi_2) + 2D$.  Switching the roles
of $\phi_1$ and $\phi_2$ and considering $h^{-1}$, for
which there is $D'>0$ with properties analogous to those of
 $D$, we see that $W^+(m, \phi_2) \le W^+(m, \phi_1) + 2D'$.  By
 the definition of $\alpha^+$ (see Equation~\ref{alpha def}), 
\[
\alpha^+(\phi_1) = \lim_{n \to \infty}\frac{W^+(n, \phi_1)}{n}
= \lim_{n \to \infty}\frac{W^+(n, \phi_2)}{n} = \alpha^+(\phi_1).
\]
The proof that $\alpha^-(\phi_1) = \alpha^-(\phi_2)$ is similar, 
and thus the asymptotic light cones of $\phi_1$ and $\phi_2$ are identical.
\end{proof}

\subsection{A complement to Theorem~\ref{thm nonexpansive}}

In Theorem~\ref{thm nonexpansive} we showed that lines in the
spacetime of an endomorphism $\phi$ which form the boundary of its asymptotic
light cone $\A(\phi)$ are nonexpansive subspaces.  In this section we want to
show that in many instances, given an arbitrary \z2s $Y$ and
a nonexpansive subspace $L \subset \R^2$ for $Y$, there is a
$\Z^2$-isomorphism $\Psi$ taking the space $Y$ to the underlying $\Z^2$-subshift of a spacetime $\U$
 of an automorphism
$\phi \in \Aut(X,\sigma)$ for some shift $(X,\sigma)$ such that $\Psi(L)$
is an edge of the asymptotic light cone $\A(\phi)$.  In
particular this holds if $Y$ has finitely many nonexpansive
subspaces. 
 Hence in that case every nonexpansive subspace in $Y$ is
(up to isomorphism) an edge of an asymptotic light cone for some
automorphism.

To do this it is useful to introduce the notion of 
{\em expansive ray} which incorporates both the subspace and its
orientation

By a {\em ray} in  $\R^2$  we mean a set $\rho \subset \R^2$ such that
there exists $w \ne 0 \in \R^2$ with
\[
\rho = \rho(w) = \{ t w\colon    t \in [0, \infty)\}. 
\]
The space of all rays in $\R^2$ is naturally homeomorphic to the set 
of unit vectors in $\R^2$, which is the circle $S^1$.

\begin{defn}
Let $Y$ be a \z2s.  We say $\rho$  is an {\em expansive
ray } for $Y$ if the line $L$ containing $\rho$ with orientation given
by $L^+ = \rho \cap (L\setminus \{0\})$ is positively expansive
(see Definition~\ref{def: sided-expansive} and the paragraph preceding it).
\end{defn}

The concept of expansive ray is essentially the same as that of 
{\em oriented expansive line} introduced in \S 3.1 of \cite{CK}.
We emphasize that this concept is defining {\em one-sided expansiveness}
for the line $L$ containing $\rho$.  Which side of $L$ codes the other is determined
by the orientation of $\rho$ and the orientation of $\Z^2$.

To relate this to our earlier notions  of expansive
(Definition~\ref{def: expansive}) observe that
if $L$ is the subspace containing $\rho$, 
then $L$ is expansive if and only if both 
$\rho$ and $-\rho$ are expansive rays.  In this terminology, 
Theorem~\ref{thm nonexpansive} says that the rays 
$\rho^+(\phi) := \{ \la \alpha^+ v,  v \ra\colon v \ge 0\}$ and 
$\rho^-(\phi) := \{ \la \alpha^- v, v \ra \colon v \le 0\}$ are nonexpansive rays.
We note that it is not in general the case that $-\rho^+(\phi)$
and $-\rho^-(\phi)$ are nonexpansive rays.

The following lemma is essentially contained in \cite{BL}, but differs from
results there in that
we consider one-sided expansiveness. In particular note the
following result implies that being positively expansive is an open condition
for oriented one-dimensional subspaces of the $\R^2$ associated to a \z2s. 
Similarly being negatively expansive is an open condition.

\begin{lemma}\label{lemma : E open}
If $\E \subset S^1$ is the set of expansive rays for 
a \z2s $Y$, then $\E$ is open.
\end{lemma}

\begin{proof}
We show that the set $\cN$  of nonexpansive rays is closed.
Suppose that $\rho_n = \{t w_n \colon t \ge 0\}_{n=1}^\infty$ is a sequence
of rays in $\R^2$ with $\displaystyle \lim_{n \to \infty}w_n = w_0 \ne 0$
so that $\rho_0$ is the limit of the rays $\rho_n, n \ge 1$.  If the
rays $\rho_n$ are nonexpansive we must show that $\rho_0$
is nonexpansive.

Let $L_n$ be the line containing $w_n$  with the orientation such that $w_n \in L^+_n $
and let $H^+(L_n)$ be the component of $\R^2 \setminus {L_n}$ such that
for all  $w' \in H^+(L_n)$ the ordered basis $\{w_n, w'\}$ is positively oriented
and let $H^-(L_n)$ be the other component of $\R^2 \setminus {L_n}$.
Define the linear function $f_n \colon \R^2 \to \R$ by $f_n(u) = u \cdot v_n$
where $v_n$ is a unit vector in $H^+(L_n)$ which is orthogonal to $w_n$.
Then we have the following:
\begin{itemize}
\item $L_n = \ker(f_n)$
\item  A vector $u$ is in $H^+(L_n)$ if and only if $f_n(u) >0$ and
in $H^-(L_n)$ if and only if $f_n(u) <0.$
\item $\displaystyle \lim_{n \to \infty}f_n(v_0)  = f_0(v_0) = 1$.
\end{itemize}
By Proposition~\ref{prop: +expansive}  we know there exist
$\eta_n, \eta'_n \in Y$ and $z_n \in \Z^2$ such that $\eta_n(v) =  \eta'_n(v)$ 
for all $v \in H^-(L_n)$ but $\eta_n(z_n) \ne  \eta'_n(z_n)$.  By shifting 
$\eta_n$ and $\eta'_n$ we may assume lengths $|z_n|$ are bounded.
Choosing a subsequence we may assume $\{z_n\}$ is constant, say,
$z_n = z_0 \in Z^2$.
Since $Y$ is compact we may further choose subsequences
$\{\eta_n\}_{n=1}^\infty$ and $\{\eta'_n\}_{n=1}^\infty$ which converge,
say to $\eta_0$ and $\eta'_0$ respectively.
Clearly $\eta_0(z_0) \ne  \eta'_0(z_0)$.  Now if
$y \in H^-(L_0) \cap \Z^2 $ then $f_0(y) <0$ so 
$f_n(y) <0$  for sufficiently large $n$ and hence
$y \in H^-(L_n) \cap \Z^2 $. It follows that
$\eta_0(y) =  \eta'_0(y)$. 

Since $\eta_0$ and  $\eta'_0$ agree on $H^-(L_n) \cap \Z^2$
but disagree at $z_0$ we conclude that $H^-(L_n) \cap \Z^2$
does not code $H^+(L_n) \cap \Z^2$ so $\rho_0$ is a 
nonexpansive ray.
\end{proof}

\begin{prop}\label{converse}
Suppose $Y$ is a \z2s and
$\E$ is the set of expansive rays for $Y$ (thought of as a subset of $S^1$). 
Suppose $C$ is a component of $\E$ and $\rho_1, \rho_2$ are the
endpoints of the open interval $C$. Then 
there exists a shift $(X,\sigma)$ with
endomorphism   $\phi$
and an isomorphism $\Psi\colon Y \to \U(\phi)$ from $Y$ to the spacetime
of $\phi$ 
(thought of as a $\Z^2$-system)
such that the lines $L_1 := \Span(\Psi(\rho_1))$ and 
$L_2 := \Span(\Psi(\rho_2))$ are
the two edges of the asymptotic light cone $\A(\phi)$ of $\phi$.
\end{prop}
Note that $\Psi$ is not a spacetime isomorphism, as the system $Y$ is not assumed to be a spacetime. 
\begin{proof}
We consider $C$ as an open interval $(\rho_1, \rho_2)$ in the
circle $S^1$ of rays in $\R^2$.
There is a \z2s  isomorphism
$\Psi_0\colon Y \to Y_0$, where $Y_0$ is a \z2s
with  $\la 1,0\ra \in \Psi_0(C)$.  
Thus the horizontal axis
with the usual orientation is a positively expansive
subspace for the \z2s $Y_0$.
We may recode $Y_0$ to $Y_1$ by an isomorphism
$\Psi_1\colon  Y_0 \to Y_1$ in such a way that the horizontal axis $H_0$  in
$\Z^2$ codes the positive half space $\{\la i, j\ra \in \Z^2\colon j > 0\}$ (this follows from Lemma 3.2 in~\cite{BL} where we recode $Y_0$ such that ``symbols'' in $Y_1$ are vertically stacked arrays of symbols from $Y_0$ of an appropriate height).  We let $\Psi\colon Y \to Y_1$ be the composition $\Psi_1 \circ \Psi_0$.

Let $X$ denote the set of colorings of $\Z$ obtained by restricting
elements  $\eta \in Y_1$ to $H_0$.  We could equally well describe
$X$ as the  colorings of $\Z$ obtained by restricting elements
of $Y$ to the horizontal row $H_{-1} := \{\la i, j\ra \in \Z^2\colon j = -1\}$
and define $\phi\colon  X \to X$ by $\phi(x) = x'$ if there is $\eta \in Y_1$ such
that $x = \eta|_{H_0}$ and $x' = \eta|_{H_{-1}}$.
Then clearly $\phi$
is an endomorphism
 and $Y_1$ is $\U(\phi)$,  the spacetime
of $\phi$.  

Note that the ray $\rho^+(\phi) := \{ \la \alpha^+ v,  v \ra \colon v \ge 0\}$
lies in the light cone $\A(\phi)$ of $\phi$ (and in the upper half space
of $\R^2$). 
If $m > \alpha^+(\phi)$ and $\rho_m$ is the ray
$\rho_m := \{ \la mv, v \ra \colon  v \ge 0\}$, then by
Proposition~\ref{prop m>alpha} $\rho_m$ is an expansive ray.
Since by Theorem~\ref{thm nonexpansive} $\rho^+(\phi)$ is not an
expansive ray, it follows that $\Psi(\rho_2) = \rho^+(\phi)$.

Letting $\rho^-(\phi) := \{ \la \alpha^- v, v \ra \colon v \le 0\}$, a similar proof
shows that $\Psi(\rho_1) = \rho^-(\phi)$. Hence the lines $L_1$ and $L_2$ form
the edges of the asymptotic light cone $\A(\phi)$.
\end{proof}

We are not able to show which lines can arise as the edges of the asymptotic light cone: 
\begin{question} 
Does there exist a subshift of finite type $X$ and an automorphism $\phi\in\Aut(X)$ such that some edge of the asymptotic light cone of $\phi$ has irrational slope?  If so, what set of angles is achievable? 
More generally, for a subshift of finite type $X$ or for a general shift $X$, what are all of the components of the expansive subspaces?
\end{question} 
Hochman~\cite{hochman} points out that, as there are only countably many shifts of finite type, this set must be countable (and, in particular, cannot contain all irrational slopes).
If $X$ is not required to be a subshift of finite type, then Hochman's results show that for the first 
question, the only constraint on the light cone for an automorphism (of an infinite subshift) comes from $-\infty< \alpha^-\leq\alpha^+<\infty$. 

\subsection{Asymptotic spread}
Let $\ell(n, \phi)$ be the minimal length of an interval $J \subset \Z$ which
contains $0$ and $\phi^n$-codes $\{0\}$ and let $\mathcal L(\phi^n)$
be the minimal length of an interval $J_0 \subset \Z$ which
is symmetric about $0$ and $\phi^n$-codes $\{0\}$.  
It is straightforward to see that both $\ell(n,\phi)$ and $\mathcal L(\phi^n)$ are 
subadditive sequences. 

\begin{defn}
Define the {\em asymptotic spread} $A(\phi)$ of $\phi \in \End(X)$ to be
\begin{equation}
\label{eq:asy-spread}
A(\phi) = \lim_{n \to \infty} \frac{\ell(n, \phi)}{n}.
\end{equation}
We say $\phi$ is {\em range distorted} if 
$A(\phi) = 0$.  
\end{defn}

Note that since the sequence $\ell(n, \phi)$ is subadditive, Fekete's Lemma implies 
that the limit in~\eqref{eq:asy-spread} exists.

The asymptotic spread is a measure of both the width of the asymptotic light cone, 
as well as how that cone deviates from the vertical.

\begin{rem}\label{rem: range distorted}
Since the function $\mathcal L(\phi^n)$ is a subadditive function of $n\ge 0,$
by Fekete's Lemma, the limit
\[
\rho(\phi) =  \lim_{n \to \infty} \frac{\mathcal L(\phi^n)}{n}
\]
exists. Clearly $\mathcal L(\phi^n) \le \ell(n, \phi) \le  2 \mathcal L(\phi^n) +1$
and so 
\[
\rho(\phi) \le A(\phi) \le   2 \rho(\phi).
\]
In particular, $\phi$ is range distorted if and only if 
\[
\lim_{n \to \infty} \frac{\mathcal L(\phi^n)}{n} = 0
\]
\end{rem}

\begin{prop}\label{prop: distort}
If $\phi  \in \Aut(X)$ and 
$\alpha^+(\phi) = \alpha^-(\phi) = \alpha^+(\phi^{-1}) = \alpha^-(\phi^{-1})$,  then
the line $u = \alpha^+(\phi)  v$ is the unique nonexpansive one-dimensional
subspace. In particular,  if $\phi, \phi^{-1} \in \Aut(X)$ are both range distorted,  then 
the vertical axis ($u = 0$) is the unique nonexpansive subspace
\end{prop}

\begin{proof} 
The first statement follows immediately from 
Theorem~\ref{thm nonexpansive} and Proposition~\ref{prop m>alpha}.  The second 
statement follows from the first, since $\phi$ and $\phi^{-1}$ are both range distorted
if and only if $\alpha^+(\phi) = \alpha^-(\phi) 
= \alpha^+(\phi^{-1}) = \alpha^-(\phi^{-1}) = 0$.
\end{proof} 

It was shown by M.~Hochman~\cite{hochman} that if $L$ is any $1$-dimensional subspace of $\R^2$, then there exists a subshift $X_L$ and an automorphism $\phi_L\in\Aut(X_L)$ such that $L$ is the unique nonexpansive subspace for the spacetime of $\phi_L$.  Moreover, the automorphisms $\phi_L$ in his examples always have infinite order (in particular, when $L$ is vertical, $\phi_L$ is range distorted and has infinite order).  However, the space $X_L$ he constructs
 lacks many natural properties one might assume about a subshift; for example, it is not a subshift of finite type 
 and it is not transitive.   He asks the following natural question: 
\begin{question}[{Hochman~\cite[Problem 1.2]{hochman}}]
Does every nonempty closed set of one-dimensional subspaces of $\R^2$ arise 
as the nonexpansive subspaces of a $\Z^2$-action that is transitive 
(or even minimal) and supports a global ergodic measure? 
\end{question} 
We do not answer this question, but recall it here as, in particular, we do not know whether a transitive subshift can have a range distorted automorphism of infinite order.  We mention further that, in the special case that $L$ is vertical, Hochman shows that his example $(X_L,\phi_L)$ is logarithmically distorted.

\begin{prop}\label{prop: spread}
If $\phi$ is an endomorphism of a subshift of finite type $(X,\sigma)$,  
then $A(\phi)$ is determined by the light cone of $\phi$ and is,
in fact, the  length of the smallest interval containing 
$0, \alpha^-(\phi)$ and $\alpha^+(\phi)$.
\end{prop}

\begin{proof}
It follows from Proposition~\ref{lem: interval-code2} that 
if $\sigma$ is a subshift of finite type, 
then for all $x \in X$ and all sufficiently large $n>0$, 
the interval $[W^-(n), W^+(n)]$  is an interval which codes 
$\phi^n(x)[0]$ and
which is contained in any interval which contains $0$ and codes $\phi^n(x)[0]$.
It follows that if $J_n$ is the smallest interval containing 
$0, W^-(n)$ and $W^+(n)$, then
\[
A(\phi) = \lim_{n \to \infty} \frac{|J_n|}{n}.
\]
Hence $A(\phi)$
is the  length of the smallest interval containing 
$0, \alpha^-(\phi)$ and $\alpha^+(\phi)$.
\end{proof}

The following result is essentially the same as Proposition
5.3 of Tisseur's paper \cite{tisseur}, except that we consider an
arbitrary $\phi \in \Aut(X, \sigma)$ with $\sigma$ an arbitrary
shift while he considers a cellular automaton defined on the full
shift and preserving the uniform measure on that shift.  Our proof is quite
short and  makes no use of measure. It makes explicit the connection
 between the topological entropy of a shift and the topological
entropy of an automorphism of that shift.

\begin{thm}\label{thm: entropy}
If  $\phi \in \End(X)$,
then 
\[
h_{\topo}(\phi) \le A(\phi) h_{\topo}(\sigma),
\]
where  $A(\phi)$ is the asymptotic spread of
$\phi$. In particular, if $\phi$ is range distorted then 
$h_{\topo}(\phi) = 0$.  
\end{thm}

\begin{proof} 
Let $\U$ be the spacetime of $\phi.$ 
For $z\in \U$, let $R_{m,n} =\{(i,j)\in\Z^2\colon0\leq i<m, 0\leq j<n\}$ 
and let $z|_{R_{m,n}}$ denote 
the restriction of $z$ to $R_{m,n}$.  Recall that $P_\U$ denotes the two dimensional 
complexity function (see Definition~\ref{def: complexity}).  
Then 
\[
h_{\topo}(\phi)=\lim_{m\to\infty}\lim_{n\to\infty}\frac{1}{n}\log(P_{\mathcal{U}}(R_{m,n})). 
\]

Since $A(\phi)$ is the length of the smallest interval containing 
$0, \alpha^-(\phi)$ and $\alpha^+(\phi)$, for a fixed $m$ there is an interval 
$J$ in $\Z$  with length $A(\phi) n + \lo(n) +m$ that
$\phi^j$-codes the block $[0,m]$ for all $0 \le j \le n$.  In other words, 
the interval $J \times \{0\} \subset \U$ codes $R_{m,n}$.
Therefore, for any $\varepsilon >0$,  and $m$ and $n$ sufficiently large,
\[
P_{\mathcal{U}}(R_{m,n})\leq P_X(A(\phi) n +\lo(n) +m) \leq (\exp(h_\sigma+\varepsilon))^{A(\phi) n +m}.
\]  
Hence $\log(P_{\mathcal{U}}(R_{m,n}))\leq (A(\phi) n +m) (h_\sigma  +\varepsilon)$ and
\[
h_{\topo}(\phi)=\lim_{m\to\infty}\lim_{n\to\infty}\frac{\log(P_{\mathcal{U}}(R_{m,n}))}{n}\leq
\lim_{m\to\infty}\lim_{n\to\infty}\frac{(A(\phi) n +m) (h_\sigma +\varepsilon)}{n}
=A(\phi) (h_\sigma+\varepsilon).
\]
Since this holds for all $\varepsilon >0$,  the desired inequality follows.

By definition $\phi$ is range distorted if and only if $A(\phi) = 0$, and so
the last two assertions of the proposition are immediate.
\end{proof} 

\subsection{Distortion and inert automorphisms}

Recall that if $(\Sigma_A, \sigma)$ is a subshift of finite type, 
there is a dimension group representation 
$\Psi\colon \Aut(\Sigma_A) \to \Aut(D_A)$ mapping automorphisms
of the shift to automorphisms of its dimension group $D_A$
(see~\cite{LM},~\cite{W}, and~\cite{BK} for definitions).
A particularly important subgroup of $\Aut(\Sigma_A)$
is $\Inert(\Sigma_A)$,  defined to be the kernel
of $\Psi.$  An automorphism $\phi \in \Aut(\Sigma_A)$ is
called {\em inert} if $\Psi(\phi) = \id$. 

There is one special case when $\Psi$ can be thought of as
a homomorphism from $\Aut(\Sigma_A)$ to the group of positive reals
under multiplication.  This occurs when $\Sigma_A$ is an irreducible
subshift of finite type and $\det(I - At)$ is an irreducible polynomial.
In this setting, one can associate to each $\phi \in \Aut(\Sigma_A)$ 
an element $\lambda_\phi = \Psi_0(\phi)$ in $(0, \infty)$ such that 
$\Psi_0$ is a homomorphism and $\lambda_\phi = 1$ if and only
if $\phi$ is inert. 

To investigate the relationship between being inert and being 
distorted, we quote the following important result of Boyle and 
Krieger: 
\begin{thm}[Boyle and Krieger~{\cite[Theorem 2.17]{BK}}]
\label{thm: BK}
Suppose $(\Sigma_A, \sigma)$ is an irreducible subshift of finite type
and  $\det(I - At)$ is an irreducible polynomial. Then if $\phi \in \Aut(\Sigma_A)$ and 
$m$ is sufficiently large, $\sigma^m \phi$ is conjugate to a subshift
of finite type and 
$$h_{\topo}(\sigma^m \phi) = \log(\lambda_\phi) +
m h_{\topo}(\sigma).$$  
\end{thm}

\begin{thm}
Suppose $(\Sigma_A, \sigma)$ is an irreducible subshift of finite type
such that $\det(I - At)$ is an irreducible polynomial, and let $\phi \in \Aut(\Sigma_A)$.
If $\phi$ and $\phi^{-1}$ are range distorted, then $\phi$ is inert.
\end{thm}

\begin{proof}
Let $\lambda_\phi = \Psi(\phi)$ and note that
by replacing $\phi$ with $\phi^{-1}$ if necessary,  we can 
assume that $\lambda_\phi \ge 1$.    Suppose $\phi$ is range distorted and so
$\alpha^+(\phi) = \alpha^-(\phi)=0$; we show that $\phi$ is inert.  
From parts~\eqref{item:one} and~\eqref{item:two} 
of Proposition~\ref{alpha prop},  we conclude that 
$\alpha^+(\sigma^k\phi) = \alpha^-(\sigma^k \phi)=k$.
By Proposition~\ref{prop: spread}, it follows that $A_{\sigma^k\phi} = |k|$.  
Hence by Theorem~\ref{thm: entropy}, we have 
$h_{\topo} (\sigma^k\phi) \le |k| h_{\topo}(\sigma)$.  Combining this with
the fact from Theorem~\ref{thm: BK} which says for large $k$ we have
$h_{\topo}(\sigma^k \phi) = \log(\lambda_\phi) +
k h_{\topo}(\sigma)$, we conclude that $\log(\lambda_\phi) \le 0$ or
$\lambda_\phi \le 1$.  Since we also have $\lambda_\phi \ge 1$, 
we conclude that $\lambda_\phi = 1$ and $\phi$ is inert.
\end{proof}

\end{document}